%% file: Latest.tex
\documentclass[a4paper,12pt,oneside]{amsart}

\usepackage[german,english]{babel}

\usepackage{amsmath,amssymb,amsthm,enumerate,mathtools}

\usepackage{dsfont}
\usepackage{color}
\usepackage[utf8]{inputenc} 
\usepackage[automark,nouppercase]{scrpage2}
\usepackage{hyperref} 
\usepackage[numbers]{natbib}

\input{basis_paper.tex}

\title[L\MakeLowercase{évy processes conditioned to avoid an interval}]{Lévy processes with finite variance conditioned to avoid an interval}

\author{Leif D\"oring}
\address{Leif D\"oring: University of Mannheim, Institute of Mathematics, 68161 Mannheim, Germany.}
\email{doering@uni-mannheim.de}

\author{Alexander R. Watson}
\address{Alexander R. Watson: University of Manchester, School of Mathematics, Manchester, M13 9PL, UK.}
\email{alex.watson@manchester.ac.uk}

\author{Philip Weißmann$^{*}$}\thanks{$^{*}$Supported by the Research Training Group "Statistical Modeling of Complex Systems" funded by the German Science Foundation}
\address{Philip Weißmann: University of Mannheim, Institute of Mathematics, 68161 Mannheim, Germany.}
\email{hweissma@mail.uni-mannheim.de}

\allowdisplaybreaks

\begin{document}
\setlength{\parindent}{0pt}

\begin{abstract}
Conditioning Markov processes to avoid a set is a classical problem that has been studied in many settings. In the present article we study the question if a L\'evy process can be conditioned to avoid an interval and, if so, the path behavior of the conditioned process. For L\'evy processes with finite second moments we show that conditioning is possible and identify the conditioned process as an $h$-transform of the original killed process. The $h$-transform is explicit in terms of successive overshoot distributions and is used to prove that the conditioned process diverges to $+\infty$ and $-\infty$ with positive probabilities.
\end{abstract}

\maketitle 
\tableofcontents
\section{Introduction}
Conditioning Markov processes to avoid sets is a classical problem. Indeed, suppose $(\mP^x)_{x\in E}$
is a family of Markov probabilities on the state space $E$, and that  $T$ is the first hitting time of a fixed set. When $T$ is almost surely finite, it is non-trivial to construct and characterise the conditioned process through the natural limiting procedure
\begin{align}\label{cond}
	&\lim_{s\to \infty}\mP^x(\Lambda\,|\, s+t<T) %\textcolor{red}{\quad \lim_{s\to \infty}\mP^x(\Lambda, t<s\,|\, s<T)}
\end{align}
or the randomized version
\begin{align}\label{cond_exp}
	&\lim_{q \to 0}\mP^x(\Lambda, t<e_q\,|\, e_q<T),
\end{align}
for $\Lambda\in \mathcal F_t$ and $x\in E$. Here, $(\mathcal{F}_t)_{t\geq 0}$ denotes the natural filtration of the underlying Markov process and $e_q$ are independent exponentially distributed random variables with parameter $q>0$. 

\smallskip

A classical example is Brownian motion conditioned to avoid the negative half-line.
In this case, the limits \eqref{cond} and \eqref{cond_exp} lead to a so-called Doob $h$-transform
of the Brownian motion killed on entering the negative half-line, by the positive
harmonic function $h(x) = x$ on $(0,\infty)$. This Doob $h$-transform
turns out (see Chapter VI.3 of \cite{Rev_Yor_01}) to be the Bessel process of dimension $3$,
which is transient. This example is typical, in that a conditioning procedure
leads to a new process which is transient where the original process was recurrent.
\smallskip

Extensions of this result have been obtained in several directions, most notably to random walks and L\'evy processes. A prominent example with several applications is that of a L\'evy process conditioned to stay positive, which was found by Chaumont and Doney \cite{Chau_Don_01} using
the randomised conditioning \eqref{cond_exp}.
In that case, the associated harmonic function $h$ is given by the potential function of the descending ladder height process. Similarly, Bertoin and Doney \cite{Bert_Don_01} have shown how to condition a random walk to stay non-negative. Other examples include random walks conditioned to stay in a cone (Denisov and Wachtel \cite{Den_Wac_01}), isotropic stable processes conditioned to stay in a cone (Kyprianou et al. \cite{Kyp_Riv_Sat_02}), spectrally negative L\'evy processes conditioned to stay in an interval (Lambert \cite{Lam_01}), subordinators conditioned to stay in an interval (Kyprianou et al. \cite{Kyp_Riv_Sen_02}), L\'evy processes conditioned to avoid the origin 
(Pant\'i  \cite{Pan_01} and Yano \cite{Yan_01}) and
self-similar Markov processes conditioned to avoid the origin (Kyprianou et al. \cite{Kyp_Riv_Sat_01}).\smallskip

The purpose of this article is to take advantage of the path discontinuities of Lévy processes and to condition them to avoid an interval. In Döring et al. \cite{Doer_Kyp_Wei_01} this problem was tackled for strictly stable processes since their structure as self-similar Markov processes allowed to deduce the right harmonic functions. The proofs were based on the so called deep factorisation (see Kyprianou et al.\ \cite{Kyp_02,Kyp_Riv_Sen_01}), which analyses stable process using the Lamperti-Kiu transform. In the present article, we consider Lévy processes with zero mean and finite variance. This assumes less structure on the L\'evy process, but at the same time excludes the stable processes, which have infinite second moments. The discrete-time analogue of our problem was considered by
Vysotsky \cite{Vys_01}, who used a Doob $h$-transform to condition a centred random walk with finite second moment to avoid an interval. One of the harmonic functions we will discover is the analogue of the harmonic function found by Vysotsky for random walks, but the techniques needed are different.\smallskip

Before presenting our results, we introduce the most important definitions and results concerning Lévy processes. More details can be found, for example, in Bertoin \cite{Bert_01}, Kyprianou \cite{Kyp_01} or Sato \cite{Sat_01}.\smallskip

\textbf{L\'evy Processes:} A Lévy process $\xi$ is a stochastic process with stationary and independent increments whose trajectories are almost surely right-continuous with left-limits (RCLL).
For each $x \in \mathbb{R}$, we define the probability measure $\mP^x$ under which
the canonical process $\xi$ starts at $x$ almost surely. We write $\mP$ for the measure $\mP^0$. The dual measure $\hat{\mP}^x$ denotes the law of the so-called dual process $-\xi$ started at $x$.
A Lévy process can be identified using its characteristic exponent $\Psi$, defined by the equation
$\mE[ e^{\mathrm{i} q \xi_t}] = e^{-t\Psi(q)}$, $q\in \mR$, which has the Lévy-Khintchine representation:
\[ \Psi(q) = \mathrm i a q + \frac{1}{2} \sigma^2q^2
  + \int_{\mathbb{R}} (1-e^{\mathrm i q x} + \mathrm i q x \1_{\{\lvert x \rvert < 1\}} ) \, \Pi(\mathrm d x),\quad q \in \mR,
\]
where $a\in\mR$ is the so-called centre of process, $\sigma^2\geq 0$ is the variance of the Brownian component, and the L\'evy measure $\Pi$ is a real measure with no atom at $0$ satisfying $\int (x^2\wedge 1) \Pi(dx)<\infty$.
\smallskip

Our main assumption is:
\begin{align*}
 (\textbf{A}) &\qquad \text{$\xi$ has zero mean and finite variance, and is not a compound Poisson process}.
\end{align*}
We define $T_B=\inf\{ t \geq 0: \xi_t \in B \}$ for any open or closed set $B \subseteq \mR$. This is known to be a stopping time with respect to the right-continuous natural enlargement of the filtration induced by $\xi$, which we denote by $(\mathcal{F}_t)_{t\geq 0}$.
For certain auxiliary results, we will need to distinguish two cases:
\begin{align*}
	(\textbf{B}) \qquad \Pi(b-a,\infty)>0,\text{ i.e., upward jumps avoiding $[a,b]$ are possible}
\end{align*}
and 
\begin{align*}
	\qquad(\hat{\textbf{B}}) \qquad \Pi(-\infty,a-b)>0,\text{ i.e., downward jumps avoiding $[a,b]$ are possible}
\end{align*}

 \smallskip

\textbf{Killed L\'evy processes and $h$-transforms:}
For $a<b$ the killed transition measures are defined as
$$p^{[a,b]}_t(x,\dd y)=\mP^x(\xi_t \in \dd y, t< T_{[a,b]}), \quad t\geq 0.$$
The corresponding sub-Markov process is called the L\'evy process killed in $[a,b]$.
A harmonic function for the killed process is a measurable function $h : \mR \backslash [a,b] \rightarrow [0,\infty)$ such that
\begin{align}\label{eq_def_harm}
	\mE^x \big[ \1_{\{ t < T_{[a,b]} \}} h(\xi_t) \big] = h(x),
	\quad x \in \mR \backslash [a,b], t \geq 0
\end{align}
A harmonic function taking only strictly positive values is called a positive harmonic function. Thanks to the Markov property, harmonicity is equivalent to $(\1_{\{ t < T_{[a,b]} \}} h(\xi_t))_{t \geq 0}$ being a $\mP^x$-martingale with respect to $(\cF_t)_{t \geq 0}$. When $h$ is a positive harmonic function,  the associated Doob $h$-transform is defined via the change of measure
\begin{align}\label{def_htrafo}
	\mP_h^x(\Lambda) := \mE^x \Big[ \1_\Lambda \1_{\{ t < T_{[a,b]} \}} \frac{h(\xi_t)}{h(x)}  \Big],\quad x\in\mR \backslash [a,b],
\end{align}
for $\Lambda \in \cF_t$.
%All probability measures will be defined on the path space of killed trajectories
%\begin{align*}
%D_{[a,b]} := \big\{ \omega:[0,\infty) \rightarrow \mR \setminus [a,b] \cup \{ \Delta \} \,  \big|&  \, \exists \zeta(\omega)\in (0,\infty] \text{ s.t. } \omega \text{ is RCLL on } [0,\zeta(\omega)),\\
%& \, \omega_t = \Delta \, \forall \, t \geq \zeta \big\}
%\end{align*}
From Chapter 11 of Chung and Walsh \cite{Chu_Wal_01}, we know that under $\mP_h^{x}$ the canonical process is a conservative strong Markov process. In Chapter 11 of Chung and Walsh \cite{Chu_Wal_01} it is shown that (\ref{def_htrafo}) extends from deterministic times to $(\cF_t)_{t \geq 0}$-stopping times $T$;
that is,
\begin{align}\label{eq_htrafo_stoppingtime}
\mP_{h}^x(\Lambda, T <\zeta) = \mE^x \Big[\1_\Lambda \1_{\{ T<T_{[a,b]} \}} \frac{h(\xi_T)}{h(x)} \Big], \quad x \notin [a,b],
\end{align}
for $\Lambda\in \mathcal F_T$.\smallskip

\textbf{Ladder height processes and potential functions:} A crucial ingredient in our analysis is the potential function $U_-$ of the descending ladder height process, which is positive harmonic for a L\'evy process killed on the negative half-line. To introduce $U_-$, some notation is needed. Denote the local time of the Markov process $(\sup_{s \leq t} \xi_s - \xi_t)_{t\ge 0}$ at $0$ by $L$, which is also called the local time of $\xi$ at the maximum. Let $L_t^{-1} = \inf\{s>0 : L_s > t\}$ denote the inverse local time at the maximum and $\kappa(q) = -\log \mE \big[ e^{- q L^{-1}_1} \big]$, for $q \geq 0$, the 
Laplace exponent of $L^{-1}$. We define $H_t =\sup_{s\leq L^{-1}_t}\xi_s$, the so-called (ascending) ladder height process. It is well-known that $H$ is a subordinator and we denote by $a_+$ the drift coefficient of $H$, and by $\mu_+$ its Lévy measure.
Under the dual measure $\hat\mP$, the process $L^{-1}$ is the inverse
local time at the minimum, and we denote its Laplace exponent by $\hat\kappa$.
Still under this dual measure, $H$ is the descending ladder height process,
and we define $a_-$ and $\mu_-$ to be its drift coefficient and Lévy measure.
\smallskip

The $q$-resolvents of $H$, for $q \geq 0$, will be denoted by $U_+^q$; that is,
$$U^q_+(\dd x) \coloneqq \mE \Big[ \int\limits_{[0,\infty)} e^{-qt} \1_{\{ H^+_t \in \dd x, L_t^{-1} <\infty\}} \, \dd t \Big].$$ 
For $q=0$ we abbreviate $U_+(\dd x) = U_+^0(\dd x)$, and denote the so-called potential function by $U_+(x) = U_+([0,x])$, for $x\geq 0$. We define $U_-^q$ and $U_-$ according to the same procedure for the descending ladder height process. If $\xi$ is not a compound Poisson process, it is known that $U_+$ and $U_-$ are continuous.

\section{Main results}\label{sec:main}
Before stating the main results, some more notation is needed to define our harmonic functions.
We first define inductively the sequence of successive stopping times at which the process jumps crossing $a$ or $b$:
\begin{align*}
	\tau_0 &\coloneqq 0,\\
	\tau_{k+1} &\coloneqq \inf\{ t > \tau_k: \xi_{t-} >b, \xi_t \leq b \} \wedge  \inf\{ t > \tau_k: \xi_{t-} <a, \xi_t \geq a \}.
\end{align*}
Second, let $K^\dag \coloneqq \inf \{ k \geq 1: \tau_k = T_{[a,b]}\}$ be the index indicating the time at which the process hits the given interval, let
$$\nu_k^x(\dd y) = \mP^x(\xi_{\tau_k} \in \dd y, \tau_k < \infty, k\leq K^{\dagger} ),
\quad x,y \in \mR \setminus [a,b]$$
be the distribution of the position of $\xi$ after its $k$-th jump across the interval,
for $k \geq 0$.
\smallskip

It is important to note that each $\nu_k^x$ can be expressed explicitly in terms of the Lévy measures and potential measures of the ladder height processes. Indeed, $\nu^x_1$ is nothing but an overshoot distribution, for which a formula is given in Proposition III.2 of Bertoin \cite{Bert_01}, using that the overshoot of $\xi$ has the same distribution as the overshoot of the corresponding ladder height subordinator $H$. Applying the strong Markov property successively yields explicit expressions for all other $\nu^x_k$.

\begin{theorem}\label{thm_harmonic}
If Assumptions $(A)$ and $(B)$ hold, then the function 
\begin{align*}
h_+(x) &:=
\begin{cases}
\sum\limits_{k=0}^\infty \int\limits_{(b,\infty)} U_-(y-b) \,\nu^x_{2k}(\dd y) & \text{if } x>b\\
\sum\limits_{k=0}^\infty \int\limits_{(b,\infty)} U_-(y-b) \,\nu^x_{2k+1}(\dd y) &  \text{if } x<a
\end{cases}
%&=\begin{cases}
%\sum\limits_{k=0}^\infty \mE^x\la U_-(\xi_{\tau_{2k}}-b)\1_{\{ K^{\dagger} \geq 2k, \tau_{2k} <\infty \}} \ra &\quad \text{if } x>b\\
%\sum\limits_{k=0}^\infty \mE^x\la U_-(\xi_{\tau_{2k+1}}-b) \1_{\{ K^{\dagger} \geq 2k+1, \tau_{2k+1} <\infty \}} \ra &\quad \text{if } x<a
%\end{cases}
\end{align*}
is a positive harmonic function for $\xi$ killed on entering $[a,b]$, i.e.
$$\mE^x \big[ \1_{\{ t < T_{[a,b]} \}} h_+(\xi_t) \big] = h_+(x), \quad t\geq 0,x \in \mR\setminus [a,b].$$
\end{theorem}

\smallskip

If Assumption $(B)$ is not satisfied, then $h_+$ is always harmonic, but may not be positive.
To be precise, when $(B)$ fails, $h_+$ is positive on $(b,\infty)$ but zero on $(-\infty,a)$.
\smallskip

Similarly, under $(A)$ and $(\hat B)$, the function
\begin{align*}
h_-(x)&=
\begin{cases}
\sum\limits_{k=0}^\infty \int\limits_{(-\infty,a)} U_+(a-y) \,\nu^x_{2k+1}(\dd y) & \text{if } x>b\\
\sum\limits_{k=0}^\infty \int\limits_{(-\infty,a)} U_+(a-y) \,\nu^x_{2k}(\dd y)  & \text{if } x<a
\end{cases}
%&=\begin{cases}
%\sum\limits_{k=0}^\infty \mE^x\la U_+(a-\xi_{\tau_{2k+1}})\1_{\{ K^{\dagger} \geq 2k+1, \tau_{2k+1} <\infty \}} \ra ,\quad \text{if } x>b\\
%\sum\limits_{k=0}^\infty \mE^x\la U_+(a-\xi_{\tau_{2k}}) \1_{\{ K^{\dagger} \geq 2k, \tau_{2k} <\infty \}} \ra  ,\quad \text{if } x<a.
%\end{cases}
\end{align*}
is positive harmonic as well. As above, when $(\hat B)$ fails,
$h_-$ remains harmonic, but is positive only on $(-\infty,a)$.
\smallskip

An important corollary of this discussion is the existence of positive harmonic functions under the Assumption $(A)$ only:
\begin{corollary}\label{thm_harmonic_h}
If Assumption $(A)$ holds, then all linear combinations of $h_+$ and $h_-$ with strictly positive coefficients are positive harmonic functions.
\end{corollary}
The harmonic functions $h_+$ and $h_-$ typically do not have a simple closed form (for a positive example see Section \ref{sec:3} below). This seemingly reduces their applicability but they are explicit enough to be used for conditioning purposes. We use them below as a tool to prove that conditioning in the sense of (\ref{cond_exp}) works and, as a consequence of general $h$-transform theory, obtain that the conditioned process is strong Markov. Additionally, it turns out that the harmonic functions are explicit enough to explain the limiting behavior of trajectories under the conditioned law. 

\begin{remark}
Vysotsky \cite{Vys_01} considered the analogous problem for a centred random walk $S=(S_n)_{n \in \mN}$ with finite variance. He derived a harmonic function $V$ which is the discrete analogue of some linear combination of $h_+$ and $h_-$. Proving harmonicity in the discrete-time situation is less involved for the following reason. It is enough to show that $V(S)$ is a discrete-time martingale for which it is enough to derive the martingale property for one time-step. Since, in discrete-time, $1 \leq T_{[a,b]}$ for $x \notin [a,b]$ the computation is direct. The continuous-time situation of L\'evy processes is much more delicate as $t\leq T_{[a,b]}$ does not hold almost surely for any $t\geq 0$.
\end{remark}
With the harmonic functions $h_+$, $h_-$ and their positive linear combinations it is now possible to $h$-transform the killed process as in \eqref{def_htrafo}.  The $h$-transforms $\mP_+$ (resp. $\mP_-$) are defined through \eqref{def_htrafo} with the positive harmonic functions $h_+$ (resp. $h_-$). In the sequel we identify the right ways to condition in order to obtain $h$-transforms with $h_+$ and $h_-$ and then derive the right linear combination of $h_+$ and $h_-$ that corresponds to conditioning the L\'evy process to avoid the interval in the sense of \eqref{cond_exp}.\smallskip
%All probability measures will be defined on the path space of killed trajectories
%\begin{align*}
%D_{[a,b]} := \big\{ \omega:[0,\infty) \rightarrow \mR \setminus [a,b] \cup \{ \Delta \} \,  \big|&  \, \exists \zeta(\omega)\in (0,\infty] \text{ s.t. } \omega \text{ is RCLL on } [0,\zeta(\omega)),\\
%& \, \omega_t = \Delta \, \forall \, t \geq \zeta \big\}
%\end{align*}

The next proposition gives a probabilistic representation of $\mP_+^x$ by conditioning to avoid $[a,b]$ and staying above $b$ at late times. The analogous conditioning under $(A)$ and ($\hat B$) below the interval results in the $h$-transform $\mP_-^x$.
\begin{proposition}\label{thm_cond_pos}
Assume $(A)$ and $(B)$.  Then
$$\mP_+^x(\Lambda) = \lim_{q \searrow 0} \mP^x(\Lambda, t < e_q \, | \, e_q <T_{[a,b]}, \xi_{e_q} >b),\quad x \notin [a,b],$$
for $\Lambda \in \cF_t$, $t \geq 0$.
\end{proposition}
To understand the L\'evy process to avoid the interval without additional condition on the late values a natural guess is an $h$-transform with a linear combination of $h_+$ and $h_-$. Possible skewness of the L\'evy process implies that different weights must be chosen for $h_+$ and $h_-$. Our proofs show that the right harmonic function is
\begin{align}\label{h}
	h:=h_++C h_-,\quad \text{where}\quad C= \lim\limits_{q \searrow 0} \frac{\kappa(q)}{\hat\kappa(q)}.
\end{align}
Note that, $\hat{\kappa}(q)$ and $\kappa(q)$ behave like $\sqrt{q}$ for $q \searrow 0$ if $\xi$ oscillates and has finite variance, see for instance Patie and Savov \cite{Pat_Sav_01}, Remark 2.21. Hence, $C$ exists and is strictly positive and from Corollary \ref{thm_harmonic_h} it follows that $h$ is a positive harmonic function if we assume only $(A)$. The $h$-transform of $\xi$ killed in $[a,b]$ with $h$ from \eqref{h} will be denoted by $\mP_\updownarrow$. Our main result can now be formulated. Conditioning to avoid an interval is always possible for L\'evy processes with second moments and the conditioned law corresponds to the $h$-transform with $h$ from \eqref{h}.
\begin{theorem}\label{thm_cond}
Assume $(A)$. Then,
$$\mP_\updownarrow^x(\Lambda) = \lim_{q \searrow 0} \mP^x(\Lambda, t < e_q \, | \, e_q <T_{[a,b]}),\quad x \notin [a,b],$$
for $\Lambda \in \cF_t$, $t \geq 0$.
\end{theorem}
Typically the first property analyzed for a conditioned process is the longtime behavior. It is often the case that the conditioning turns a recurrent process into a transient process. Nonetheless, a priori it is completely unclear what the limit behavior under $\mP_\pm$ and in particular $\mP_\updownarrow$ is. Processes might be oscillating, diverge to $+\infty$ or $-\infty$, or might even diverge to both infinities with positive probability. The next proposition covers the case $\mP_+$:
\begin{proposition}\label{thm_drift+}
Assume $(A)$ and $(B)$. Then $\mP_{+}^x(\lim\limits_{t \rightarrow \infty} \xi_t = +\infty)=1$ for all $x \notin [a,b].$
\end{proposition}
Analogously, assuming $(A)$ and $(\hat B)$ one can show that $\xi$ drifts to $-\infty$ almost surely under $\mP^x_{-}$. It remains to consider the behaviour of $(\xi,\mP^x_\updownarrow)$. Our final theorem shows that L\'evy processes with second moments conditioned to avoid an interval drift to $+\infty$ and $-\infty$ with (explicit) positive probabilities:
\begin{theorem}\label{thm_drifth}
Assume $(A)$. Then, $\mP^x_\updownarrow$ is transient in the sense that
$$\mE^x_\updownarrow \Big[\int\limits_{[0,\infty)} \1_{\{ \xi_t \in K \}} \, \dd t \Big] <\infty, \quad x \notin [a,b],$$
for all bounded $K \subseteq \mR\setminus [a,b]$. More precisely, 
$$\mP^x_\updownarrow\big(\lim\limits_{t \rightarrow \infty} \xi_t = +\infty\big) = \frac{h_+(x)}{h(x)} \qquad \text{and} \qquad \mP^x_\updownarrow\big(\lim\limits_{t \rightarrow \infty} \xi_t = -\infty\big) = \frac{Ch_-(x)}{h(x)}, \quad x \notin [a,b],$$
so that, in particular, $\mP^x_\updownarrow$-almost surely trajectories do not oscillate.
\end{theorem}
In the recent article \cite{Doer_Kyp_Wei_01} it was proved that stable processes conditioned to avoid an interval are transient. Since stable processes have infinite second moments our new results do not apply and it remains unclear if trajectories oscillate or diverge to $+\infty$ and $-\infty$ with positive probabilities. One can even use explicit formulas for the potential functions and the overshoot distributions (see e.g. Rogozin, \cite{Rog_01}) to show that in the stable case $h_+(x)$ and $h_-(x)$ are infinite for all $x \notin [a,b]$.

\section{An explicit example}\label{sec:3}

When $\xi$ is a Lévy process with no drift and two-sided exponential jumps, 
it is possible to compute the harmonic functions $h_+$, $h_-$ and $h$ explicitly. Let
\begin{align}\label{example_model}
	\xi_t = \sigma B_t + \sum_{i=1}^{N_t}Y_i,\quad t \geq 0,
\end{align}
where $\sigma \geq 0$, $(B_t)_{t \geq 0}$ is a standard Brownian motion,
and $\sum_{i=1}^{N_t}Y_i$ is a compound Poisson process with rate $\lambda>0$ and absolutely continuous jump distribution with density
$$f_Y(y) = \frac{1}{2} e^{-\eta y} \1_{\{y>0\}} + \frac{1}{2} e^{-\eta (-y)} \1_{\{y<0\}}.$$
For definiteness, let $\sigma=\sqrt{2}$ and $\lambda=1$. The Laplace exponent $\psi$ of $\xi$,
given by $\mE[e^{-\theta \xi_t}] = e^{-t\psi(\theta)}$, can be expressed,
for $\theta \in (-\eta,\eta)$, by
\begin{align}\label{ddd}
  \psi(\theta) 
  &= -\theta^2 - \frac{\theta^2}{(\eta-\theta)(\eta+\theta)}
  = \frac{\theta(\beta+\theta)}{\eta+\theta} \cdot \frac{(-\theta)(\beta-\theta)}{\eta-\theta}.
\end{align}
where $\beta = \sqrt{\eta^2+1} > \eta$.
Note that $\xi$ oscillates and has finite variance, so $(A)$ holds, $(B)$ and $(\hat B)$ both hold as well. Let
\[ \upsilon(\theta) = \hat\upsilon(\theta) = \frac{\theta(\beta+\theta)}{\eta+\theta}
  = \theta + (\beta-\eta) \int_0^\infty (1-e^{-\theta x})\eta e^{-\eta x}\, \mathrm{d} x,
  \qquad \theta > -\eta,
\]
which is the Laplace exponent of a subordinator with unit drift, jump rate $\beta-\eta$
and exponential jumps of parameter $\eta$. Since
\[
  \psi(\theta) = \upsilon(\theta)\hat\upsilon(-\theta),
\]
the uniqueness of the Wiener--Hopf factorisation \cite[Theorem 6.15(iv)]{Kyp_01}
implies that $\upsilon$ and $\hat\upsilon$ are indeed the Laplace exponents of the
ascending and descending ladder height subordinators, respectively.

Since
\begin{align}
\int\limits_{[0,\infty)} e^{-\theta x} \, U_-(\dd x)=\int\limits_{[0,\infty)} e^{-\theta x} \, U_+(\dd x) = \frac{1}{\psi_+(\theta)} = \frac{\eta+\theta}{\theta(\beta+\theta)}
\end{align}
by \cite[equation (5.23)]{Kyp_01},
we can identify the potential measures
\begin{align*}
U_-(\dd x) = U_+(\dd x) = \Big(\frac{\eta}{\beta} + \frac{\beta-\eta}{\beta} e^{-\beta x} \Big) \, \dd x.
\end{align*}
and the potential functions
\begin{align} \label{example_U}
U_-(x)= U_+(x) = \frac{\eta}{\beta}x + \frac{\beta-\eta}{\beta^2}(1-e^{-\beta x}), \quad x \geq 0.
\end{align}

\smallskip

To find $h_+$ in closed form we first need to find the measures $\nu_k^x$ explicitly.
This can in principle be done using the expressions we have just found for $U_\pm$ 
and the Lévy measures of the ladder height subordinators, but in fact
the overshoot distributions have already been found in
Kou and Wang \cite{Kou_Wan_01}, Corollary 3.1, where
$$\mP^x(\xi_{T_{[a,\infty)}} \in \dd y) = \frac{\eta(\beta-\eta)}{\beta}(1-e^{-\beta(a-x)}) e^{-\eta(y-a)}, \quad x<a<y,$$
and
$$\mP^x(\xi_{T_{(-\infty,b]}} \in \dd y) = \frac{\eta(\beta-\eta)}{\beta}(1-e^{-\beta(x-b)}) e^{-\eta(b-y)}, \quad x>b>y,$$
are proven.

We now claim that
\begin{align}\label{claim1}
\nu^x_{2k+1}(\dd y) = c^{2k} \nu_1^x(\dd y),\quad x<a, y>b,
\end{align}
and
\begin{align}\label{claim2}
\nu^x_{2k+2}(\dd y) = c^{2k} \nu_2^x(\dd y), \quad x,y>b,
\end{align}
hold for all $k\ge 0$, 
where $c = e^{-\eta(b-a)}(\beta-\eta)/(\beta+\eta)$. For proving this, note that
\begin{align*}
\int\limits_{(b,\infty)} (1-e^{-\beta(z-b)})e^{-\eta(z-a)} \, \dd z &= \int\limits_{(-\infty,a)} (1-e^{-\beta(a-z)})e^{-\eta(b-z)} \, \dd z\\
 &= e^{-\eta(b-a)} \frac{\beta}{\eta(\beta+\eta)}.
\end{align*}
For $k=0$ the claims are clearly correct. Next, note that for $x>b$:
\begin{align*}
\nu^x_2(\dd y) &= \int\limits_{(-\infty,a)} \mP^z(\xi_{T_{[a,\infty)}} \in \dd y) \, \mP^x(\xi_{T_{(-\infty,b]}} \in \dd z) \\
&=\Big(\frac{\eta(\beta-\eta)}{\beta}\Big)^2 (1-e^{-\beta(x-b)})e^{-\eta(y-a)} \int\limits_{(-\infty,a)}  (1-e^{-\beta(a-z)}) e^{-\eta(b-z)} \dd z \\
&=\Big(\frac{\eta(\beta-\eta)}{\beta}\Big)^2 (1-e^{-\beta(x-b)})e^{-\eta(y-a)} e^{-\eta(b-a)} \frac{\beta}{\eta(\beta+\eta)}\\
&= c \frac{\eta(\beta-\eta)}{\beta} (1-e^{-\beta(x-b)})e^{-\eta(y-a)}.
\end{align*}

\smallskip

Now, let us assume the claims are correct for $k-1$, $k \geq 1$. Then, for $x<a,b<y$,
\begin{align*}
\nu^x_{2k+1}(\dd y) &= \int\limits_{(b,\infty)} \nu^z_{2k}(\dd y)\, \mP^x(\xi_{T_{[a,\infty)}} \in \dd z) \\
&= c^{2k-2} \int\limits_{(b,\infty)} \nu^z_{2}(\dd y) \,\mP^x(\xi_{T_{[a,\infty)}} \in \dd z)\\
&= c^{2k-2} c \Big(\frac{\eta(\beta-\eta)}{\beta}\Big)^2 (1-e^{-\beta(a-x)}) e^{-\eta(y-a)} \int\limits_{(b,\infty)}(1-e^{-\beta(z-b)})  e^{-\eta(z-a)} \, \dd z \\
&= c^{2k-1} \Big(\frac{\eta(\beta-\eta)}{\beta}\Big)^2 e^{-\eta(b-a)} \frac{\beta}{\eta(\beta+\eta)} (1-e^{-\beta(a-x)}) e^{-\eta(y-a)} \\
&= c^{2k-1}  \Big(\frac{\beta-\eta}{\beta+\eta}\Big)e^{-\eta(b-a)} \mP^x(\xi_{T_{[a,\infty)}} \in \dd y) \\
&= c^{2k} \nu_1^x(\dd y),
\end{align*}
which is \eqref{claim1}. Similarly we get, for $x,y>b$,
\begin{align*}
\nu^x_{2k+2}(\dd y) &= \int\limits_{(-\infty,a)} \nu^z_{2k+1}(\dd y)\, \mP^x(\xi_{T_{(-\infty,b]}} \in \dd z) \\
&= c^{2k} \int\limits_{(-\infty,a)}  \nu_1^z(\dd y)\, \mP^x(\xi_{T_{(-\infty,b]}} \in \dd z)\\
&= c^{2k} \int\limits_{(-\infty,a)}  \nu_1^z(\dd y)\, \nu_1^x(\dd z)\\
&= c^{2k}  \nu_2^x(\dd y)
\end{align*}
which is \eqref{claim2}.\smallskip

Having formulas for $U_-$ and all $\nu_k$ we can proceed to compute $h_+$. Combining \eqref{example_U}, \eqref{claim1} and \eqref{claim2} standard integration shows, for $k \geq 1$,
\begin{align*}
\int\limits_{(b,\infty)} U_-(y-b) \, \nu_{2k+1}^x(\dd y) &= c^{2k} \int\limits_{(b,\infty)} U_-(y-b) \, \nu_{1}^x(\dd y)\\
&= c^{2k} \frac{2c}{\beta} (1-e^{\-\beta(a-x)})\\
&= \frac{2c^{2k+1}}{\beta} (1-e^{\-\beta(a-x)})
\end{align*}
for $x<a$ and
\begin{align*}
\int\limits_{(b,\infty)} U_-(y-b) \, \nu_{2k+2}^x(\dd y) &= c^{2k} \int\limits_{(b,\infty)} U_-(y-b) \, \nu_{2}^x(\dd y)\\
&= c^{2k} \frac{2c^2}{\beta} (1-e^{-\beta(x-b)})\\
&= \frac{2c^{2k+2}}{\beta} (1-e^{-\beta(x-b)})
\end{align*}
for $x>b$. Hence, plugging-into the definition of $h_+$ gives
\begin{align*}
h_+(x) &= \Big(\sum\limits_{k=0}^\infty c^{2k+1} \Big) \frac{2}{\beta} (1-e^{-\beta(a-x)})= \frac{2c}{\beta(1-c^2)} (1-e^{-\beta(a-x)})
\end{align*}
for $x<a$ and 
\begin{align*}
h_+(x) &= \frac{\eta}{\beta}(x-b) + \frac{\beta-\eta}{\beta^2}(1-e^{-\beta (x-b)}) + \Big(\sum\limits_{k=0}^\infty c^{2k+2} \Big) \frac{2}{\beta} (1-e^{-\beta(x-b)})\\
&= \frac{\eta}{\beta}(x-b) + \frac{\beta-\eta}{\beta^2}(1-e^{-\beta (x-b)}) + \frac{2c^2}{\beta(1-c^2)} (1-e^{-\beta(x-b)})\\
&= \frac{\eta}{\beta}(x-b) + \Big(\frac{\beta-\eta}{\beta^2} + \frac{2c^2}{\beta(1-c^2)} \Big)(1-e^{-\beta(x-b)})
\end{align*}
for $x>b$. Analogously we obtain 
\begin{align*}
	h_-(x)=
	\begin{cases}
	 \frac{2c}{\beta(1-c^2)} (1-e^{-\beta(x-b)})	&\text{if } x >b \\
	 \frac{\eta}{\beta}(a-x) + \Big(\frac{\beta-\eta}{\beta^2} + \frac{2c^2}{\beta(1-c^2)} \Big)(1-e^{-		\beta(a-x)})	&\text{if } x <a \\
	\end{cases}
\end{align*}
and, finally, 
\begin{align*}
	h(x)= 
	\begin{cases}
	  \frac{\eta}{\beta}(x-b) + \Big(\frac{\beta-\eta}{\beta^2} + \frac{2(c+c^2)}{\beta(1-c^2)} \Big) (1-e^{-\beta(x-b)})	&\text{if } x >b \\
	 \frac{\eta}{\beta}(a-x) + \Big(\frac{\beta-\eta}{\beta^2} + \frac{2(c+c^2)}{\beta(1-c^2)} \Big)(1-e^{-		\beta(a-x)})	&\text{if } x <a \\
	\end{cases}.
\end{align*}
We used that by symmetry $\kappa = \hat\kappa$ and consequently $C=\lim_{q \searrow 0} \kappa(q)/\hat\kappa(q)=1$.

\section{Proofs}\label{sec_proofs}
Before going into the proofs let us discuss the form of the measures $\nu_k$ defined before. We assume in the theorems that $\xi$ oscillates, hence, all appearing first hitting times are almost surely finite. Keeping in mind that on the event $\{K^\dag > k\}$ the time $\tau_k$ is the time of the $k^{th}$ jump across the interval. By the strong Markov property and $\nu_0^x(\dd y) = \delta_x(\dd y)$, we find the relations 
\begin{align*}
\nu_{2k+1}^x(\dd y) &= \int\limits_{(b,\infty)} \mP^z(\xi_{T_{(-\infty,b]}} \in \dd y) \, \nu^x_{2k}(\dd z)=\int\limits_{(b,\infty)} \, \nu_1^z(\dd y) \, \nu_{2k}^x(\dd z), \\
\nu_{2k}^x(\dd y) &= \int\limits_{(-\infty,a)} \mP^z(\xi_{T_{[a,\infty)}} \in \dd y) \, \nu^x_{2k-1}(\dd z)=\int\limits_{(-\infty,a)} \, \nu_1^z(\dd y) \, \nu_{2k-1}^x(\dd z),
\end{align*}
for $x>b$, and
\begin{align*}
\nu_{2k+1}^x(\dd y) &= \int\limits_{(-\infty,a)} \mP^z(\xi_{T_{[a,\infty)}} \in \dd y) \, \nu^x_{2k}(\dd z)=\int\limits_{(-\infty,a)} \, \nu_1^z(\dd y) \, \nu_{2k}^x(\dd z),\\
\nu_{2k}^x(\dd y) &= \int\limits_{(b,\infty)} \mP^z(\xi_{T_{(-\infty,b]}} \in \dd y) \, \nu^x_{2k-1}(\dd z)=\int\limits_{(b,\infty)} \, \nu_1^z(\dd y) \, \nu_{2k-1}^x(\dd z),
\end{align*}
for $x<a$. More generally, the strong Markov property also implies the relation
\begin{align}\label{cc}
\int\limits_{(b,\infty)} \, \nu_l^z(\dd y) \, \nu_{2k}^x(\dd z) = \nu_{2k+l}^x(\dd y) \qquad \text{and} \qquad
\int\limits_{(-\infty,a)} \, \nu_l^z(\dd y) \, \nu_{2k+1}^x(\dd z) = \nu_{2k+l+1}^x(\dd y)
\end{align}
for $x>b$ and $k,l \in \mN$ and the analogous identities hold for $x<a$. It is important to note that (see e.g. Bertoin \cite{Bert_01}, Proposition III.2) analytic formulas exist for the overshoot distributions:
\begin{align}\label{overshoot}
\mP^x(\xi_{T_{[a,\infty)}} \in \dd y) = \int\limits_{[x,a]} \, \mu_+(\dd y - u) \, U_+(\dd u-x), \quad x<a<y,
\end{align}
and, analogously,
\begin{align}\label{overshootdown}
\mP^x(\xi_{T_{(-\infty,b]}} \in \dd y) = \int\limits_{[b,x]} \, \mu_-(u-\dd y) \, U_-(x-\dd u), \quad x>b>y.
\end{align}
Hence, analytic expressions for the $\nu_k$ exist in the oscillating case even though these become more involved for big $k$ due to the recursive definition. As an example, for $x>b$, we have
\begin{align*}
\nu_2^x(\dd y) &= \int\limits_{(-\infty,a)} \, \mP^z(\xi_{T_{[a,\infty)}} \in \dd y) \, \mP^x(\xi_{T_{(-\infty,b]}} \in \dd z)\\
&= \int\limits_{(-\infty,a)} \Big[ \int\limits_{[b,x]} \Big( \int\limits_{[x,a]} \, \mu_+(\dd y - u) \, U_+(\dd u-x) \Big) \, \mu_-(w-\dd z) \Big] \,U_-(x-\dd w).
\end{align*}

\subsection{Finiteness of the harmonic function}\label{sec_finite}
Since $h_+$ and $h_-$ are defined by infinite series finiteness has to be proved. Along the way we deduce upper bounds that are needed in the sections below.\smallskip

Note that Assumption $(A)$ implies that $\mE\la H_1 \ra$ and $\hat{\mE}\la H_1 \ra$ are finite and this will be crucial for the technical steps which are necessary to prove the following. 
\begin{proposition}\label{lemma_finite}
Assume $(A)$, then there are constants $c_1,c_2,c_3 \geq 0$ such that
$$h_+(x) \leq c_1 U_-(x-b) \1_{\{ x>b \}} + c_2 U_+(a-x) \1_{\{ x<a \}} + c_3, \quad x \notin [a,b],$$
in particular $h_+(x)$ is finite for all $x \in \mR \setminus [a,b]$.
\end{proposition}

Before we start with the proof, we need a lemma which is intuitively clear, but needs a certain argumentation:
\begin{lemma}\label{lemma_osc}
Let $\xi$ be a Lévy process which is not the negative of a subordinator. Then, for all $y,z >0$,
$$\mP(T_{(-\infty,-y]} > T_{[z,\infty)})>0.$$
\end{lemma}
\begin{proof}
Assume $\mP(T_{(-\infty,-y]} \leq T_{[z,\infty)}) = 1$. Then it follows of course that 
$\mP^x(T_{(-\infty,x-y]} \leq T_{[z,\infty)}) = 1$ for all $x<0$. With the Markov property we get, for $s > 0$,
\begin{align*}
\mP(T_{[z,\infty)} < s) &= \mE\Big[ \mP^{\xi_{T_{(-\infty,-y]}}}(T_{[z,\infty)} < s) \Big] \\
&\leq  \mP^{-y}(T_{[z,\infty)} < s) \\
&= \mE^{-y}\Big[ \mP^{\xi_{T_{(-\infty,-2y]}}}(T_{[z,\infty)} < s) \Big] \\
&\leq \mP^{-2y}(T_{[z,\infty)} < s).
\end{align*}
Inductively we get $\mP(T_{[z,\infty)} < s) \leq \mP^{-ny}(T_{[z,\infty)} < s)$
for all $n \in \mN$ and hence
$$\mP(T_{[z,\infty)} < s) \leq \lim_{n \rightarrow \infty} \mP^{-ny}(T_{[z,\infty)} < s) = 0.$$
With this we see
\begin{align*}
\mP(T_{[z,\infty)} < +\infty) = \lim_{s \rightarrow \infty} \mP(T_{[z,\infty)} < s)
\leq  \lim_{s \rightarrow \infty} \lim_{n \rightarrow \infty} \mP^{-ny}(T_{[z,\infty)} < s)
= 0,
\end{align*}
but this cannot happen unless $\xi$ is the negative of a subordinator. This concludes the proof.
\end{proof}

To prove Proposition \ref{lemma_finite} we will combine two statements. The discrete analogous statements were also used (with different arguments) by Vysotsky \cite{Vys_01} to show finiteness of the harmonic function in the discrete case.
\begin{lemma}\label{lemma_help_1}
  Suppose that $\mE[H_1]<\infty$, then
  $$\gamma_+ := \sup_{x<a}\mP^x(\xi_{T_{[a,\infty)}}>b,T_{[a,\infty)}<\infty ) <1.$$ 
\end{lemma}
\begin{proof}
If $\xi$ is the negative of a subordinator it holds $\gamma_+=0$. So assume that $\xi$ is not the negative of a subordinator, in particular we can apply Lemma \ref{lemma_osc}.

\smallskip

We separate three regions of the range of $x$. First we consider very small $x$, i.e. we consider the limit of $x$ tending to $-\infty$, then we consider the values of $x$ which are close to $a$ and last we treat the remaining values.

\smallskip

We begin with $x$ close to $-\infty$. If $\xi$ drifts to $-\infty$, then
$\mP^x(T_{[a,\infty]}< \infty) \to 0$ as $x \searrow -\infty$, and in particular
$\mP^x(\xi_{T_{[a,\infty)}}>b, T_{[a,\infty)}<\infty) \to 0$ also. Therefore there exist
a $K < a$ and a $\gamma_1<1$ such that $\mP^x(\xi_{T_{[a,\infty)}}>b, T_{[a,\infty)}<\infty) <\gamma_1$ when $x \leq K$.

If $\xi$ oscillates or drifts to $\infty$, the bound for $x$ close to $-\infty$ is more
involved.
Because $\mE\la H_1 \ra <\infty$, $\xi$ has stationary overshoots in the sense that the weak limit of $\mP^x(\xi_{T_{[a,\infty)}} \in \dd y)$ for $x \searrow -\infty$ exists. It can be expressed as
\begin{align}\label{d}
\wlim\limits_{x \searrow -\infty} \mP^x(\xi_{T_{[a,\infty)}} \in \dd y) = \frac{1}{\mE\la H_1 \ra} (a_+\delta_a(\dd y) + \bar{\mu}_+(y-a) \dd y),
\end{align}
where $a_+$ is the drift of $(H,\mP)$ and $\mu_+$ its Lévy measure with the right-tail $\bar\mu_+$. For the first special version of a subordinator see for example Bertoin et al. \cite{Bert_Har_Ste}, for the general version for example Bertoin and Savov \cite{Bert_Sav_01}. Since weak convergence is equivalent to the pointwise convergence of the distribution function at continuity points, due to the explicit formula in \eqref{d} it holds that, for $b>a$,
\begin{align*}
\lim\limits_{x \rightarrow -\infty} \mP^x(\xi_{T_{[a,\infty)}}>b) &= \frac{1}{\mE\la H_1 \ra} \int\limits_{(b,\infty)}  \bar{\mu}_+(y-a) \dd y\\
&= \frac{1}{\mE\la H_1 \ra} \int\limits_{(b-a,\infty)} \bar{\mu}_+(y) \dd y\\
&<  \frac{1}{\mE\la H_1 \ra} \int\limits_{(0,\infty)} \bar{\mu}_+(y) \dd y\\
&\leq 1
\end{align*}
Hence, also in this case there exist a $K<a$ and a $\gamma_1<1$ such that
$$\mP^x(\xi_{T_{[a,\infty)}}>b) \leq \gamma_1$$
for all $x \leq K$. Now we have to treat the case $x \in (K,a)$. Therefore we separate two cases.

\smallskip

Case $1$: The process $\xi$ is regular upwards. First, we consider the limit for $x \rightarrow a$. Since $\xi$ is regular upwards it holds
$$\lim\limits_{x \rightarrow a} \mP^x(\xi_{T_{[a,\infty)}} >b, T_{[a,\infty)}<\infty) < 1$$
and hence, there is some $\delta>0$ such that
$$\gamma_2:=\sup\limits_{x \in (a-\delta,a)} \mP^x(\xi_{T_{[a,\infty)}} >b,T_{[a,\infty)}< \infty) <1.$$
It remains to consider $x \in (K,a-\delta]$. First note that
\begin{align*}
&\mP^x(\xi_{T_{[a,\infty)}}>b, T_{[a,\infty)}<\infty)\\
&= \mP^x(\xi_{T_{[a,\infty)}}>b, T_{(-\infty,K]}< T_{[a,\infty)}<\infty) +\mP^x(\xi_{T_{[a,\infty)}}>b, T_{(-\infty,K]}> T_{[a,\infty)}).
\end{align*}
For the first term we use the Markov property to get
\begin{align*}
\mP^x(\xi_{T_{[a,\infty)}}>b, T_{(-\infty,K]}< T_{[a,\infty)}<\infty) &= \mE^x\Big[\1_{\{ T_{(-\infty,K]}< T_{[a,\infty)}<\infty \}} \mP^{\xi_{T_{(-\infty,K]}}}(\xi_{T_{[a,\infty)}}>b) \Big] \\
&\leq \gamma_1 \mP^x(T_{(-\infty,K]}< T_{[a,\infty)}<\infty)\\
&\leq \gamma_1 \mP^x(T_{(-\infty,K]}< T_{[a,\infty)}).
\end{align*}
Together we have for all $x \in (K,a-\delta]$:
\begin{align*}
\mP^x(\xi_{T_{[a,\infty)}}>b, T_{[a,\infty)}<\infty) &\leq \sup\limits_{x \in (K,a-\delta]} \left(\mP^x(\xi_{T_{[a,\infty)}}>b, T_{(-\infty,K]}< T_{[a,\infty)}<\infty) \right.\\
&\quad \quad\quad \quad\quad \quad \left. +\mP^x(\xi_{T_{[a,\infty)}}>b, T_{(-\infty,K]}> T_{[a,\infty)})\right) \\
&\leq \sup\limits_{x \in (K,a-\delta]} \left(\gamma_1 \mP^x(T_{(-\infty,K]}< T_{[a,\infty)}) +\mP^x(T_{(-\infty,K]}> T_{[a,\infty)})\right) \\
&=: \gamma_3.
\end{align*}
With Lemma \ref{lemma_osc} we get
$$\sup\limits_{x \in (K,a-\delta)} \mP^x(T_{(-\infty,K]} > T_{[a,\infty)}) =  \mP^{a-\delta}(T_{(-\infty,K]} > T_{[a,\infty)}) <1$$
or, equivalently,
$$\inf\limits_{x \in (K,a-\delta)} \mP^x(T_{(-\infty,K]} < T_{[a,\infty)})>0.$$
Because of this it follows that
$$\gamma_3 < \sup\limits_{x \in (K,a-\delta)} \left( \mP^x(T_{(-\infty,K]} < T_{[a,\infty)}) +  \mP^x(T_{(-\infty,K]} > T_{[a,\infty)})\right) = 1.$$

Case $2$: The process $\xi$ is not regular upwards. In this case it holds
$$\sup\limits_{x \in (K,a)} \mP^x(T_{[a,\infty)}<T_{(-\infty,K]})<1 .$$
or equivalently
\begin{align}\label{help33}
\inf\limits_{x \in (K,a)} \mP^x(T_{(-\infty,K]}<T_{[a,\infty)})>0 .
\end{align}
We split up again \begin{align*}
&\mP^x(\xi_{T_{[a,\infty)}}>b, T_{[a,\infty)}<\infty)\\
&= \mP^x(\xi_{T_{[a,\infty)}}>b, T_{(-\infty,K]}< T_{[a,\infty)}<\infty) +\mP^x(\xi_{T_{[a,\infty)}}>b, T_{(-\infty,K]}> T_{[a,\infty)}).
\end{align*}
For the first term we use the Markov property to get
\begin{align*}
\mP^x(\xi_{T_{[a,\infty)}}>b, T_{(-\infty,K]}< T_{[a,\infty)}<\infty) &= \mE^x\Big[\1_{\{ T_{(-\infty,K]}< T_{[a,\infty)}<\infty \}} \mP^{\xi_{T_{(-\infty,K]}}}(\xi_{T_{[a,\infty)}}>b) \Big] \\
&\leq \gamma_1 \mP^x(T_{(-\infty,K]}< T_{[a,\infty)}<\infty)\\
&\leq \gamma_1 \mP^x(T_{(-\infty,K]}< T_{[a,\infty)}).
\end{align*}
Together we have for all $x \in (K,a)$:
\begin{align*}
\mP^x(\xi_{T_{[a,\infty)}}>b, T_{[a,\infty)}<\infty) &\leq \sup\limits_{x \in (K,a)} \left(\mP^x(\xi_{T_{[a,\infty)}}>b, T_{(-\infty,K]}< T_{[a,\infty)}< \infty) \right.\\
&\quad \quad\quad \quad\quad \quad \left. +\mP^x(\xi_{T_{[a,\infty)}}>b, T_{(-\infty,K]}> T_{[a,\infty)})\right) \\
&\leq \sup\limits_{x \in (K,a)} \left(\gamma_1 \mP^x(T_{(-\infty,K]}< T_{[a,\infty)}) +\mP^x(T_{(-\infty,K]}> T_{[a,\infty)})\right) \\
&=: \gamma_3.
\end{align*}
From \eqref{help33} follows that
$$\gamma_3 < \sup\limits_{x \in (K,a)} \left( \mP^x(T_{(-\infty,K]} < T_{[a,\infty)}) +  \mP^x(T_{(-\infty,K]} > T_{[a,\infty)})\right) = 1.$$

\smallskip

For the general case (both, regular upwards and not) set $\gamma_+:= \max(\gamma_1,\gamma_2,\gamma_3)<1$.
\end{proof}

Analogously to the lemma before it holds
$$\gamma_- := \sup_{x>b}\mP^x(\xi_{T_{(-\infty,b]}}<a, T_{(-\infty,b]} < \infty) <1,$$
provided that $\hat\mE[H_1]<\infty$. The second Lemma which we need to prove Proposition \ref{lemma_finite} is the following:
\begin{lemma}\label{lemma_help_2}
Assume $\xi$ oscillates and $\hat{\mE}\la H_1 \ra<\infty$. For all $\alpha \in (0,1)$ there exists a constant $C_+(\alpha)>0$ such that
$$\mE^x\Big[ U_+(a-\xi_{T_{(-\infty,b]}}) \1_{\{ \xi_{T_{(-\infty,b]}} <a \}} \Big] \leq \alpha U_-(x-b) + C_+(\alpha)$$
for all $x>b$.
\end{lemma}
\begin{proof}
We start to show that
$$\int\limits_{(K,\infty)} U_+(y) \, \mu_-(\dd y) < +\infty$$
for all $K>0$. For that we estimate $U_+(y)$ for $y > K$ with Proposition III.1 of Bertoin \cite{Bert_01} which says that there are constants $c_1,c_2 \geq 0$ such that
$$U_+(x) \leq c_1 \left( \Phi\left(\frac{1}{x}\right) \right)^{-1} \quad \text{and} \quad \Phi(x) \geq c_2 x \left(I\left(\frac{1}{x}\right) + a_+\right)$$
for all $x >0$, where $\Phi(\lambda) = \mE\big[ \int_{[0,\infty)} e^{-\lambda H_t} \, \dd t \big]$ and $I(x) = \int_{(0,x]} \bar{\mu}_+(y) \, \dd y$. We combine these two statements as follows:
\begin{align*}
U_+(x) \leq c_1 \Big( \Phi\big(\frac{1}{x}\big) \Big)^{-1} 
\leq c_1 \Big( c_2 \frac{1}{x} (I(x) + a_+)\Big)^{-1} 
= \frac{c_1}{c_2} \frac{x}{I(x) + a_+} 
\leq \frac{c_1}{c_2} \frac{x}{I(K)} 
= c_K x
\end{align*}
for all $x>K$, where $c_K = \frac{c_1}{c_2 I(K)}$. Hence, by assumption,
\begin{align*}
\int\limits_{(K,\infty)} U_+(y) \, \mu_-(\dd y) &\leq c_K \int\limits_{(K,\infty)} y \, \mu_-(\dd y)
\leq c_K\hat{\mE}\la H_1 \ra 
< +\infty
\end{align*}
for all $K>0$. The second inequality can be seen from
$\hat{\mE}\la H_1 \ra = \int\limits_{(0,\infty)} y \, \mu_-(\dd y) + a_-$
because $H$ is a subordinator. Now, for fixed $\alpha \in (0,1)$, choose $K=K(\alpha)>0$ such that
\begin{align}\label{gl_K}
\int\limits_{(K,\infty)} U_+(y) \, \mu_-(\dd y) < \alpha.
\end{align}
To prove the claim let us first split as
\begin{align*}
&\quad\mE^x\Big[ U_+(a-\xi_{T_{(-\infty,b]}}) \1_{\{ \xi_{T_{(-\infty,b]}} <a \}} \Big] \\
&= \mE^x\Big[ U_+(a-\xi_{T_{(-\infty,b]}}) \1_{\{ \xi_{T_{(-\infty,b]}} \in [a-K,a) \}} \Big] + \mE^x\Big[ U_+(a-\xi_{T_{(-\infty,b]}}) \1_{\{ \xi_{T_{(-\infty,b]}} \in (-\infty,a-K) \}} \Big]
\end{align*}
and estimate the first summand, using monotonicity of $U_+$, as
\begin{align*}
\mE^x\Big[ U_+(a-\xi_{T_{(-\infty,b]}}) \1_{\{ \xi_{T_{(-\infty,b]}} \in [a-K,a) \}} \Big] &\leq U_+(K).
\end{align*}
Applying the overshoot formula \eqref{overshootdown} the second summand can be treated in the following way:
\begin{align*}
&\quad \mE^x\Big[ U_+(a-\xi_{T_{(-\infty,b]}}) \1_{\{ \xi_{T_{(-\infty,b]}} \in (-\infty,a-K) \}} \Big] \\
&= \int\limits_{(-\infty,a-K)} U_+(a-y) \, \mP^x(\xi_{T_{(-\infty,b]}} \in \dd y) \\
&=\int\limits_{[b,x]} \Big( \int\limits_{(-\infty,a-K)} U_+(a-y) \, \mu_-(w-\dd y) \Big)\, U_-(x-\dd w)\\
&= \int\limits_{[b,x]} \Big( \int\limits_{(K+w-a,\infty)} U_+(y-w+a) \, \mu_-(\dd y) \Big)\, U_-(x-\dd w) \\
&\leq \int\limits_{[b,x]} \Big( \int\limits_{(K,\infty)} U_+(y) \, \mu_-(\dd y) \Big) \, U_-(x-\dd w) \\
&\leq \alpha U_-(x-b).
\end{align*}

Defining $C_+(\alpha)\coloneqq U_+(K)$ we proved
$$\mE^x\big[ U_+(a-\xi_{T_{(-\infty,b]}}) \1_{\{ \xi_{T_{(-\infty,b]}} <a \}} \big] \leq \alpha U_-(x-b) + C_+(\alpha)$$
for all $x >b$.
\end{proof}
Analogously to the lemma above one can show in the case that $\xi$ oscillates and $\mE \la H_1 \ra <\infty$ that for all $\alpha \in (0,1)$ there exists a constant $C_-(\alpha)>0$ such that
$$\mE^x\big[ U_-(\xi_{T_{[a,\infty)}}- b) \1_{\{ \xi_{T_{[a,\infty)}} >b \}} \big] \leq \alpha U_+(a-x) + C_-(\alpha),\quad x<a.$$

Now we are ready to combine Lemmas \ref{lemma_help_1} and \ref{lemma_help_2} to show finiteness of $h_+(x)$. The idea how to combine them was also used by Vysotsky \cite{Vys_01}.
\begin{proof}[Proof of Proposition \ref{lemma_finite}]
Let $\alpha \in (0,1)$ be abitrary. In the first step we use the finiteness of $\mE\la H_1 \ra$ and $\hat\mE \la H_1 \ra$ combined with Lemmas \ref{lemma_help_1} and \ref{lemma_help_2} to find an upper bound for 
$$ \int\limits_{(b,\infty)} U_-(y-b) \, \nu^x_{2k}(\dd y),\quad x>b.$$
 Set $\gamma = \max(\gamma_+,\gamma_-)$ and note by Lemma \ref{lemma_help_1} that for $x>b$ and $k \geq 1$:
\begin{align*}
\nu^x_{2k-1}(-\infty,a) &= \int\limits_{(b,\infty)} \mP^y(\xi_{T_{(-\infty,b]}} <a) \, \nu^x_{2k-2}(\dd y) \\
&\leq \gamma \nu^x_{2k-2}(b,\infty) \\
&= \gamma \Big(\1_{\{k=1\}} + \1_{\{k\geq 2\}} \int\limits_{(-\infty,a)} \mP^y(\xi_{T_{[a,\infty)}} >b) \, \nu^x_{2k-3}(\dd y)\Big) \\
&\leq \gamma \Big(\1_{\{k=1\}} + \gamma \1_{\{k\geq 2\}} \nu^x_{2k-3}(-\infty,a)\Big).
\end{align*} 
Inductively we get 
$$\nu^x_{2k-1}(-\infty,a)\leq \gamma^{2k-1}$$
for $x>b$ and $k \geq 1$. Analogously for $k \geq 1$ we can show
$$\nu^x_{2k}(b,\infty) \leq \gamma^{2k}$$
for $x>b$ and
$$\nu^x_{2k-1}(b,\infty) \leq \gamma^{2k-1} \quad \text{and} \quad \nu^x_{2k}(-\infty,a) \leq \gamma^{2k-1}$$
for $x<a$. Now set $C(\alpha) = \max(C_-(\alpha),C_+(\alpha))$ and use Lemma \ref{lemma_help_2} for $k \geq 1$ to find
\begin{align*}
\int\limits_{(b,\infty)} U_-(y-b) \, \nu^x_{2k}(\dd y) &= \int\limits_{(-\infty,a)} \Big( \int\limits_{(b,\infty)} U_-(y-b) \, \mP^v(\xi_{T_{[a,\infty)}} \in \dd y) \Big) \, \nu^x_{2k-1}(\dd v) \\
&\leq \int\limits_{(-\infty,a)} \alpha U_+(a-v)\nu^x_{2k-1}(\dd v) +C(\alpha) \, \nu^x_{2k-1}(-\infty,a) \\
&\leq \alpha \int\limits_{(-\infty,a)}  U_+(a-v) \, \nu^x_{2k-1}(\dd v) +C(\alpha)  \gamma^{2k-1}.
\end{align*}
We estimate the first term in the same way by
$$\alpha^2\int\limits_{(b,\infty)}  U_-(b-y)\nu^x_{2k-2}(\dd y) + C(\alpha) \alpha  \gamma^{2k-2}$$
and hence,
$$\int\limits_{(b,\infty)} U_-(y-b) \, \nu^x_{2k}(\dd y) \leq \alpha^2\int\limits_{(b,\infty)}  U_-(b-y)\,\nu^x_{2k-2}(\dd y) + C(\alpha) (\gamma^{2k-1} +  \alpha  \gamma^{2k-2}).$$
Going on with this procedure until $\nu^x_0$ we see
\begin{align*}
\int\limits_{(b,\infty)} U_-(y-b) \, \nu^x_{2k}(\dd y) &\leq U_-(x-b) \alpha^{2k} + C(\alpha) \sum_{i=0}^{2k-1} \gamma^i \alpha^{2k-1-i} \\
 &= U_-(x-b) \alpha^{2k} + C(\alpha) \alpha^{2k-1} \sum_{i=0}^{2k-1} \left(\frac{\gamma}{\alpha}\right)^i  .
\end{align*}
Now note 
$$\alpha^{2k-1} \sum_{i=0}^{2k-1} \left(\frac{\gamma}{\alpha}\right)^i = \alpha^{2k-1} \frac{\big(\frac{\gamma}{\alpha}\big)^{2k}-1}{\frac{\gamma}{\alpha}-1} = \frac{\gamma^{2k}-\alpha^{2k}}{\gamma-\alpha}$$
and hence
\begin{align*}
	\int\limits_{(b,\infty)} U_-(y-b) \, \nu^x_{2k}(\dd y)  &\leq U_-(x-b) \alpha^{2k} + \frac{C(\alpha)}{\gamma-\alpha} (\gamma^{2k}-\alpha^{2k})
\end{align*}
for $k\geq 1$ (for $k=0$ we get obiously $U_-(x-b)$ as upper bound). In the same way we get for $x<a$:
$$\int\limits_{(b,\infty)} U_-(y-b) \, \nu^x_{2k+1}(\dd y) \leq U_+(a-x) \alpha^{2k+1} + \frac{C(\alpha)}{\gamma-\alpha} (\gamma^{2k+1}-\alpha^{2k+1})$$
for $k \geq 0$ (here we get an upper bound dependend on $U_+$ because the number of steps is odd). All together we get
\begin{align*}
&\quad h_+(x) \\
&\leq \1_{(b,\infty)}(x) U_-(x-b) \sum_{k=0}^\infty \alpha^{2k}  +   \1_{(-\infty,a)}(x) U_+(a-x) \sum_{k=0}^\infty \alpha^{2k+1} + \frac{C(\alpha)}{\gamma-\alpha} \sum_{k=0}^\infty (\gamma^k-\alpha^k) \\
&=  \frac{1}{1-\alpha^2} U_-(x-b)\1_{(b,\infty)}(x)  +   \frac{\alpha}{1-\alpha^2} U_+(a-x) \1_{(-\infty,a)}(x) +  \frac{C(\alpha)}{\gamma-\alpha}\big(\frac{1}{1-\gamma}-\frac{1}{1-\alpha}\big)
\end{align*}
which finishes the proof of Proposition \ref{lemma_finite}.
\end{proof}

\subsection{Harmonicity of \texorpdfstring{$h_+$}{h\textunderscore +} and \texorpdfstring{$h_-$}{h\textunderscore -}}
In this section we give the proof of Theorem \ref{thm_harmonic}. Define, for $q \geq 0$ and $x \notin [a,b]$, the auxiliary functions
\begin{align*}
h^q_+(x) &:=
\begin{cases}
\sum\limits_{k=0}^\infty \int\limits_{(b,\infty)} U^q_-(y-b) \,\nu^x_{2k}(\dd y) &\quad \text{if } x>b\\
\sum\limits_{k=0}^\infty \int\limits_{(b,\infty)} U^q_-(y-b) \,\nu^x_{2k+1}(\dd y) & \quad \text{if } x<a
\end{cases}\\
&=\begin{cases}
\sum\limits_{k=0}^\infty \mE^x\Big[ U^q_-(\xi_{\tau_{2k}}-b)\1_{\{ K^{\dagger} \geq 2k, \tau_{2k} <\infty \}} \Big] &\quad \text{if } x>b\\
\sum\limits_{k=0}^\infty \mE^x\Big[ U^q_-(\xi_{\tau_{2k+1}}-b) \1_{\{ K^{\dagger} \geq 2k+1, \tau_{2k+1} <\infty \}} \Big] &\quad \text{if } x<a
\end{cases},
\end{align*}
where $U^q_-(\dd x) \coloneqq \hat{\mE} \big[ \int_{[0,\infty)} e^{-qt} \1_{\{ H_t \in \dd x, L_t^{-1} <\infty\}} \, \dd t \big]$ is the $q$-potential of the dual ladder height process. It follows immediately that $h_+^q(x) \leq h_+(x)$ for all $x \notin [a,b]$ and by monotone convergence that $h_+^q$ converges pointwise to $h_+$ for $q \searrow 0$.
\begin{proposition}\label{prop_asymp}
Assume $(A)$ and let $e_q$ be independent exponentially distributed random variables with parameter $q>0$. Then, for $x \notin [a,b]$,
\begin{align}
	\frac{1}{\hat\kappa(q)} \mP^x\big(e_q < T_{[a,b]}, \xi_{e_q} >b\big) \leq h_+^q(x), \quad q > 0,
\end{align}
and
\begin{align}
\lim_{q \searrow 0} \frac{1}{\hat\kappa(q)} \mP^x\big(e_q < T_{[a,b]}, \xi_{e_q} >b\big) = h_+(x).
\end{align}
\end{proposition}

To prove this crucial proposition we need a small lemma which is basically just the strong Markov property:
\begin{lemma} \label{lemma_mp}
Let be $s \geq 0$ and $k \geq 0$. Then it holds
$$\int\limits_{(b,\infty)}  \mP^y (s < T_{(-\infty,b]}) \, \nu^x_{2k}(\dd y) = \mP^x \left( s < \tau_{2k+1} -  \tau_{2k}, K^{\dagger} \geq 2k+1 \right)$$
and
$$ \int\limits_{(-\infty,a)} \mP^y (s <T_{[a,\infty)}) \, \nu^x_{2k+1}(\dd y)= \mP^x \left( s < \tau_{2k+2} -  \tau_{2k+1}, K^{\dagger} \geq 2k+2 \right)$$
for $x>b$ and 
$$\int\limits_{(-\infty,a)}  \mP^y (s < T_{[a,\infty)}) \, \nu^x_{2k}(\dd y) = \mP^x \left( s < \tau_{2k+1} -  \tau_{2k}, K^{\dagger} \geq 2k+1 \right)$$
and
$$ \int\limits_{(b,\infty)} \mP^y (s <T_{(-\infty,b]}) \, \nu^x_{2k+1}(\dd y)= \mP^x \left( s < \tau_{2k+2} -  \tau_{2k+1}, K^{\dagger} \geq 2k+2 \right)$$ for $x<a$.
\end{lemma}
\begin{proof}
We focus on the case $x>b$ and prove the first equality. We use the strong Markov property in the shift operator formulation, see e.g. Chung and Walsh \cite{Chu_Wal_01}, p. 57. Therefore we introduce
$D := \{ \omega:[0,\infty) \rightarrow \mR \,|  \, \omega \text{ is RCLL} \}$. The shift operator is a map $\theta_t: D \rightarrow D$ such that $X_s \circ \theta_t = X_{t+s}$. The strong Markov property tells that for a $(\cF_t)_{t \geq 0}$-stopping time $T$ it holds 
\begin{align}\label{strongMP}
\1_{\{T<\infty\}} \mE^{\xi_{T}} [ Y ] = \1_{\{T<\infty\}} \mE^x \big[ Y \circ \theta_{T} \, |\, \cF_{T} \big]
\end{align}
for all $\cF_\infty := \bigcup_{t\geq 0} \cF_t$-measurable and integrable $Y$. Here, we set $T= \tau_{2k}$ and $Y = \1_{\{ s < T_{(-\infty,b]} \}}$. It is clear that $Y$ is bounded and that $Y$ is $\cF_\infty$-measurable can be seen as follows:
$$\{ s < T_{(-\infty,b]} \} = \{ T_{(-\infty,b]} \leq s \}^{\mathrm{C}} \in \cF_s \subseteq \cF_\infty.$$
With \eqref{strongMP} we obtain for our choice of $Y$:
$$\mP^{\xi_{\tau_{2k}}} \big(s < T_{(-\infty,b]}\big) = \mE^x \big[ \1_{\{ s < T_{(-\infty,b]} \}} \circ \theta_{\tau_{2k}} \, |\, \cF_{\tau_{2k}} \big].$$
Using this we get
\begin{align*}
&\quad\int\limits_{(b,\infty)}  \mP^y (s < T_{(-\infty,b]}) \, \nu^x_{2k}(\dd y) \\
&=  \mE^x \la \1_{\{ \xi_{\tau_{2k}} > b, K^{\dagger} \geq 2k \}}  \mP^{\xi_{\tau_{2k}}} (s < T_{(-\infty,b]})  \ra \\
&= \mE^x \la \1_{\{\xi_{\tau_{2k}} > b, K^{\dagger} \geq 2k \}}  \mE^x \big[ \1_{\{ s < T_{(-\infty,b]} \}} \circ \theta_{\tau_{2k}} \, |\, \cF_{\tau_{2k}} \big]  \ra \\
&= \mE^x \la \1_{\{ \tau_{2k} < T_{[a,b]} \}}  \mP^x(s+\tau_{2k} <  \tau_{2k+1} \, |\, \cF_{\tau_{2k}}) \ra \\
&= \mE^x \la\mP^x(\tau_{2k} < T_{[a,b]}, s < \tau_{2k+1}-\tau_{2k} \, |\, \cF_{\tau_{2k}}) \ra \\
&=\mP^x(\tau_{2k} < T_{[a,b]}, s < \tau_{2k+1}-\tau_{2k} ) \\
&= \mP^x(K^{\dagger} \geq 2k+1, s < \tau_{2k+1}-\tau_{2k} ).
\end{align*}
We used that $\{\xi_{\tau_{2k}} > b \} \in \cF_{\tau_{2k}}$ and $\{ \tau_{2k} < T_{[a,b]} \} \in \cF_{\tau_{2k}} \cap  \cF_{T_{[a,b]}} \subseteq \cF_{\tau_{2k}}$ which can be seen by Theorem 1.3.6 of \cite{Chu_Wal_01}. The remaining claims follow analogously.
%We define the event
%$$A_n = \{ \xi_{\frac{ks}{2^n}} > b : k=1,\ldots,2^n \}.$$
%It should be clear that $A_{n+1} \subseteq A_{n}$ and because of right-continuity it holds
%$$\bigcap\limits_{n \in \mN} A_n =  \{ s < T_{(-\infty,b]} \}.$$
%Applying dominated convergence (with domitaning function 1) in the second and in the fourth equality and the strong Markov property (see e.g. \cite{Sat_01} Cor. 40.11) in the third we get
%\begin{align*}
% \int\limits_{(b,\infty)}  \mP^y (s < T_{(-\infty,b]}) \, \nu^x_{2k}(\dd y) &=  \mE^x \la \1_{\{ \xi_{\tau_{2k}} > b, \tau_{2k} <\infty, K^{\dagger} \geq 2k \}}  \mP^{\xi_{\tau_{2k}}} (s < T_{(-\infty,b]})  \ra \\
% &= \lim_{n \rightarrow \infty} \mE^x  \la \1_{\{ \xi_{\tau_{2k}} > b, \tau_{2k} <\infty, K^{\dagger} \geq 2k \}}  \mP^{\xi_{\tau_{2k}}} (A_n)  \ra \\
%&= \lim_{n \rightarrow \infty} \mE^x \la \1_{\{ \xi_{\tau_{2k}} > b, \tau_{2k} <\infty, K^{\dagger} \geq 2k \}}  \1_{\{ \xi_{\tau_{2k}+\frac{ks}{2^n}} > b : \, k=1,\ldots,2^n \}} \ra \\
%&= \mE^x \la \1_{\{ \xi_{\tau_{2k}} > b, \tau_{2k} <\infty, K^{\dagger} \geq 2k \}} \lim_{n \rightarrow \infty} \1_{\{ \xi_{\tau_{2k}+\frac{ks}{2^n}} > b : \, k=1,\ldots,2^n \}} \ra \\
%&= \mP^x \left( \tau_{2k} +s < \tau_{2k+1}, K^{\dagger} \geq 2k+1 \right)
%\end{align*}
\end{proof}

Now we continue the proof of Proposition \ref{prop_asymp} for which we use the identity
\begin{align}\label{33}
	\hat{\kappa}(q)U^q_-(x) = \mP^x(e_q < T_{(-\infty,0]}),\quad x >0, q > 0,
\end{align}
proved by Kyprianou \cite{Kyp_01}, Section 13.2.1 for a general Lévy process.

\begin{proof}[Proof of Proposition \ref{prop_asymp}]
We only consider the case $x>b$ and start to prove the bounds
\begin{align}\label{tttt}
	1\leq \frac{\hat{\kappa}(q) h_+^q(x)}{\mP^x(e_q < T_{[a,b]},\xi_{e_q}>b)} \leq \frac{1}{\mP^x(e_q \geq T_{[a,b]})}.
\end{align}
To derive the lower bound we define 
$\tilde{\tau_k} = \min(\tau_k,T_{[a,b]}).$
It follows, in particular, that $\tilde{\tau}_k=\tau_k$ on $K^{\dagger} \geq k$ and $\tilde{\tau}_{k+1}-\tilde{\tau}_k=0$ on $K^{\dagger} \leq k$. For the next chain of equalities we use \eqref{33}, Lemma \ref{lemma_mp} and the lack of memory property of $e_q$:
\begin{align*}
\hat{\kappa}(q) \int\limits_{(b,\infty)} U^q_-(y-b) \, \nu^x_{2k}(\dd y) &= \int\limits_{(b,\infty)} \mP^y(e_q < T_{(-\infty,b]})\, \nu^x_{2k}(\dd y)\\
&= \mP^x(\tau_{2k+1}-\tau_{2k} > e_q, K^{\dagger} \geq 2k+1)\\
&= \mP^x(\tilde{\tau}_{2k+1}-\tilde{\tau}_{2k} > e_q) \\
&= \mP^x( \tilde{\tau}_{2k+1} > e_q | e_q \geq \tilde{\tau}_{2k}) \nonumber \\
&= \frac{ \mP^x( e_q \in [\tilde{\tau}_{2k},\tilde{\tau}_{2k+1}))}{ \mP^x( e_q \geq \tilde{\tau}_{2k})}.
\end{align*}
Furthermore, it holds that
$$\mP^x( e_q \geq \tilde{\tau}_{2k}) \ge \mP^x( e_q \geq T_{[a,b]})$$
because $\tilde{\tau}_{2k} \leq T_{[a,b]}$. So we obtain
\begin{align}\label{eq} 
\mP^x( e_q \in [\tilde{\tau}_{2k},\tilde{\tau}_{2k+1}))\leq \hat{\kappa}(q) \int\limits_{(b,\infty)} U^q_-(y-b) \, \nu^x_{2k}(\dd y) \leq  \frac{ \mP^x( e_q \in [\tilde{\tau}_{2k},\tilde{\tau}_{2k+1}))}{ \mP^x( e_q \geq T_{[a,b]})}.
\end{align}
Before proving the bounds of \eqref{tttt} we note that
\begin{align}\label{eq1}
\begin{split}
\mP^x(e_q < T_{[a,b]}, \xi_{e_q}>b) &= \mP^x(e_q < \lim_{k \rightarrow \infty} \tilde{\tau}_k,  \xi_{e_q}>b)\\
&= \mP^x\Big(\bigcup_{k=0}^\infty \{e_q \in [\tilde{\tau}_k,\tilde{\tau}_{k+1}),  \xi_{e_q}>b \}\Big)\\
&= \mP^x\Big(\bigcup_{k=0}^\infty \{e_q \in [\tilde{\tau}_{2k},\tilde{\tau}_{2k+1})\}\Big)\\
&= \sum_{k=0}^\infty \mP^x( e_q \in [\tilde{\tau}_{2k},\tilde{\tau}_{2k+1})).
\end{split}
\end{align}
The first equality follows from the definition of $\tilde{\tau}_k$ and the facts that $T_{[a,b]}<\infty$ almost surely (because $\xi$ is recurrent under Assumption (A)) and that $\tau_k$ diverges to $+\infty$ almost surely. The third one is due to the fact that for $x<b$ the process remains above $b$ only in the intervals $[\tilde{\tau}_{2k},\tilde{\tau}_{2k+1})$. With \eqref{eq1}, summing \eqref{eq} over $k$ yields 
\begin{align*}
\hat{\kappa}(q) h^q_+(x) &= \sum_{k=0}^\infty \hat{\kappa}(q)\int\limits_{(b,\infty)} U^q_-(y-b) \, \nu^x_{2k}(\dd y)\\
&\in \Big[\mP^x(e_q < T_{[a,b]}, \xi_{e_q}>b),   \frac{ \mP^x(e_q < T_{[a,b]}, \xi_{e_q}>b)}{ \mP^x( e_q \geq T_{[a,b]})} \Big]
\end{align*}
which is \eqref{tttt}. Since $\xi$ is recurrent $\mP^x(e_q \geq T_{[a,b]})$ converges to $1$ for $q \searrow 0$, hence, \eqref{tttt} implies the claim.
%\begin{align*}
%&\leq \hat{\kappa}(q) \sum_{k=0}^\infty \int\limits_{(b,\infty)} U^q_-(y-b) \, \nu^x_{2k}(\dd y) \\
%&= \hat{\kappa}(q) h_+^q(x).
%\end{align*}
%\footnote{erste gleichung ist klar, aber warum die zweite?} In total we proved
%\begin{align}\label{tt}
%	\frac{\hat{\kappa}(q) h_+^q(x)}{\mP^x(e_q < T_{[a,b]}, \xi_{e_q}>b)} \geq 1.
%\end{align}
%Next, we derive the upper bound from \eqref{tttt}. Using the upper bound of (\ref{eq}) we estimate as follows:
%\begin{align}\label{ttt}
%\begin{split}
%\hat{\kappa}(q) h_+^q(x) &= \hat{\kappa}(q) \sum_{k=0}^\infty \int\limits_{(b,\infty)} U^q_-(y-b) \, \nu^x_{2k}(\dd y) \\
%&\leq \frac{ \mP^x(e_q < T_{[a,b]}, \xi_{e_q} >b)}{\mP^x(e_q \geq T_{[a,b]})}.
%\end{split}
%\end{align}
%Combining \eqref{tt} and \eqref{ttt} yields \eqref{tttt}. Since $\xi$ is recurrent $\mP^x(e_q \geq T_{[a,b]})$ converges to $1$ for $q \searrow 0$, hence, the proof is complete.
\end{proof}

The key for the proof of Theorem \ref{thm_harmonic} are the relations in Proposition \ref{prop_asymp}. We use them in a similar way Chaumont and Doney \cite{Chau_Don_01} proved harmonicity of a certain function for the Lévy process killed on the negative half-line.
\begin{proof}[Proof of Theorem \ref{thm_harmonic}]
First note that $(B)$ guarantees that $h_+(x)$ is strictly positive for all $x \in \mR \setminus [a,b]$, which is not the case for $x<a$ when $(B)$ fails. From now on Assumption $(B)$ won't be used anymore. For $x \in \mR\setminus [a,b]$ and $t \geq 0$ we have to show
$$\mE^x\big[ \1_{\{ t < T_{[a,b]} \}} h_+(\xi_t) \big] = h_+(x).$$
First we show that the left-hand side is smaller or equal to the right-hand side. This can be done applying Proposition \ref{prop_asymp} in the first step and Fatou's Lemma in the second one:
\begin{align}\label{eq_proof_harm}
\mE^x\big[ \1_{\{ t < T_{[a,b]} \}} h_+(\xi_t) \big] &= \mE^x\Big[ \1_{\{ t < T_{[a,b]} \}} \lim_{q \searrow 0} \frac{1}{\hat\kappa(q)} \mP^{\xi_t}(e_q < T_{[a,b]}, \xi_{e_q} > b) \Big] \nonumber \\
&\leq \lim_{q \searrow 0} \frac{1}{\hat\kappa(q)} \mE^x\Big[ \1_{\{ t < T_{[a,b]} \}} \mP^{\xi_t}(e_q < T_{[a,b]}, \xi_{e_q} > b) \Big]\\
&=\lim_{q \searrow 0} \frac{q}{\hat\kappa(q)} \int\limits_{(0,\infty)} e^{-qs}  \mE^x\Big[ \1_{\{ t < T_{[a,b]} \}} \mP^{\xi_t}(s < T_{[a,b]}, \xi_{s} > b) \Big] \, \dd s \nonumber \\
&=\lim_{q \searrow 0} \frac{q}{\hat\kappa(q)} \int\limits_{(0,\infty)} e^{-qs}  \mP^{x}(s+t < T_{[a,b]}, \xi_{s+t} > b) \, \dd s \nonumber \\
&=\lim_{q \searrow 0} \frac{q}{\hat\kappa(q)} e^{qt} \int\limits_{(t,\infty)} e^{-qs}  \mP^{x}(s < T_{[a,b]}, \xi_{s} > b) \, \dd s \nonumber \\
&=\lim_{q \searrow 0} \frac{q}{\hat\kappa(q)} e^{qt} \int\limits_{(0,\infty)} e^{-qs}  \mP^{x}(s < T_{[a,b]}, \xi_{s} > b) \, \dd s \nonumber \\
&\quad - \lim_{q \searrow 0} \frac{q}{\hat\kappa(q)} e^{qt} \int\limits_{(0,t]} e^{-qs}  \mP^{x}(s < T_{[a,b]}, \xi_{s} > b) \, \dd s \nonumber \\
&=\lim_{q \searrow 0} \frac{1}{\hat\kappa(q)} e^{qt} \mP^{x}(e_q < T_{[a,b]}, \xi_{e_q} > b) \nonumber \\
&\quad - \lim_{q \searrow 0} \frac{q}{\hat\kappa(q)} e^{qt} \int\limits_{(0,t]} e^{-qs}  \mP^{x}(s < T_{[a,b]}, \xi_{s} > b) \, \dd s \nonumber \\
&=h_+(x) - \lim_{q \searrow 0} \frac{q}{\hat\kappa(q)} e^{qt} \int\limits_{(0,t]} e^{-qs}  \mP^{x}(s < T_{[a,b]}, \xi_{s} > b) \, \dd s \nonumber \\
&= h_+(x).\nonumber
\end{align}
The last equality follows because, according to Kyprianou \cite{Kyp_01}, Section 13.2.1, it holds that $\lim_{q \searrow 0} \frac{q}{\hat\kappa(q)}=0$ if $\xi$ oscillates. To show the equality it remains to show that we can replace the inequality in (\ref{eq_proof_harm}) by an equality. To apply the dominated convergence theorem, we use Proposition \ref{prop_asymp} which says also that
$$\frac{1}{\hat\kappa(q)} \mP^{\xi_t}(e_q < T_{[a,b]}, \xi_{e_q} > b) \leq h^q_+(\xi_t) \leq h_+(\xi_t)$$
for all $q > 0$. Furthermore, we have just seen that 
$$\mE^x\big[ \1_{\{ t < T_{[a,b]} \}} h_+(\xi_t) \big] \leq h_+(x) <\infty.$$
So we can apply dominated convergence to switch the limit and the integral.
\end{proof}

\subsection{Conditioning and \texorpdfstring{$h$}{h}-transforms}
The aim of this section is to prove Proposition \ref{thm_cond_pos} and Theorem \ref{thm_cond}.

\begin{proof}[Proof of Proposition \ref{thm_cond_pos}]
Integrating out $e_q$, using Proposition \ref{prop_asymp} and the Markov property, gives
\begin{align*}
& \lim\limits_{q \searrow 0} \mP^x(\Lambda, t < e_q \,|\, e_q<T_{[a,b]}, \xi_{e_q}>b) \\
=& \lim\limits_{q \searrow 0} \frac{1}{\mP^x(e_q < T_{[a,b]}, \xi_{e_q}>b)} \int\limits_{(t,\infty)}q e^{-qs} \mP^x\left(\Lambda, s <T_{[a,b]}, \xi_{s}>b \right) \, \dd s  \\
=& \frac{1}{h_+(x)}\lim\limits_{q \searrow 0} \frac{e^{-qt}}{\hat\kappa(q)} \int\limits_{(0,\infty)}q e^{-qs} \mP^x\left(\Lambda, s+t <T_{[a,b]}, \xi_{s+t}>b \right) \, \dd s  \\
=& \frac{1}{h_+(x)}\lim\limits_{q \searrow 0} \frac{1}{\hat\kappa(q)} \int\limits_{(0,\infty)}q e^{-qs} \mE^x\Big[ \1_\Lambda \1_{\{ t <T_{[a,b]} \}} \mP^{\xi_t}(s<T_{[a,b]},\xi_{s}>b) \Big] \, \dd s  \\
=& \frac{1}{h_+(x)}\lim\limits_{q \searrow 0} \frac{1}{\hat\kappa(q)} \mE^x\Big[ \1_\Lambda \1_{\{ t <T_{[a,b]} \}} \mP^{\xi_t}(e_q<T_{[a,b]},\xi_{e_q}>b) \Big].
\end{align*}
From Proposition \ref{prop_asymp} we also know $\frac{1}{\hat\kappa(q)}\mP^{\xi_t}(e_q<T_{[a,b]},\xi_{e_q}>b) \leq h_+(\xi_t)$ for all $q > 0$ and $\1_\Lambda \1_{\{ t <T_{[a,b]} \}} h_+(\xi_t)$ is integrable since $h_+$ is harmonic. So we can use dominated convergence to conclude
\begin{align*}
& \lim\limits_{q \searrow 0} \mP^x(\Lambda, t < e_q | e_q<T_{[a,b]}, \xi_{e_q}>b) \\
=& \frac{1}{h_+(x)}\mE^x\Big[ \1_\Lambda \1_{\{ t <T_{[a,b]} \}} \lim\limits_{q \searrow 0} \frac{1}{\hat\kappa(q)}  \mP^{\xi_t}(e_q<T_{[a,b]},\xi_{e_q}>b) \Big] \\
=& \mE^x\Bigg[ \1_\Lambda \1_{\{ t <T_{[a,b]} \}} \frac{h_+(\xi_t)}{h_+(x)} \Bigg],
\end{align*}
where we used again Proposition \ref{prop_asymp} in the final equality. Hence, conditioning is possible and coincides with the $h$-transform with $h_+$ which confirms Proposition \ref{thm_cond_pos}.
\end{proof}

For the proof of Theorem \ref{thm_cond} we will use a corollary of Proposition \ref{prop_asymp}.
\begin{corollary}\label{cor_asymp}
Assume $(A)$ and let $e_q$ be an independent exponentially distributed random variable with parameter $q>0$. Then, for $x \notin [a,b]$, we have
\begin{align}
\mP^x(e_q < T_{[a,b]}) \leq \hat\kappa(q) h_+^q(x) + \kappa(q) h_-^q(x), \quad q > 0,
\end{align}
and
\begin{align}
	\lim_{q \searrow 0} \frac{1}{\hat\kappa(q)} \mP^x(e_q < T_{[a,b]}) = h_+(x) + C h_-(x),
\end{align}
where $C = \lim_{q \searrow 0} \frac{\kappa(q)}{\hat\kappa(q)}$.
\end{corollary}
\begin{proof}
Let be $x \notin [a,b]$. With Proposition \ref{prop_asymp} and its counterpart for $h_-$ we have
$$\mP^x(e_q < T_{[a,b]},\xi_{e_q}>b) \leq \hat\kappa(q) h_+^q(x) \quad \text{ and } \quad \mP^x(e_q < T_{[a,b]},\xi_{e_q}<a) \leq \kappa(q) h_-^q(x)$$
from which the first claim follows. Furthermore we have again with Proposition \ref{prop_asymp}:
$$\lim_{q \searrow 0} \frac{1}{\hat\kappa(q)} \mP^x(e_q < T_{[a,b]},\xi_{e_q}>b) = h_+(x)$$
and
$$ \lim_{q \searrow 0} \frac{1}{\kappa(q)} \mP^x(e_q < T_{[a,b]},\xi_{e_q}<a) = h_-(x).$$
With this we get
\begin{align*}
&\quad \lim_{q \searrow 0} \frac{1}{\hat\kappa(q)} \mP^x(e_q < T_{[a,b]})\\
&= \lim_{q \searrow 0}\frac{1}{\hat\kappa(q)} \mP^x(e_q < T_{[a,b]},\xi_{e_q}>b) + \lim_{q \searrow 0} \frac{\kappa(q)}{\hat\kappa(q)} \frac{1}{\kappa(q)} \mP^x(e_q < T_{[a,b]},\xi_{e_q}<a)\\
&= h_+(x) + C h_-(x)
\end{align*}
and the proof is complete.
\end{proof}

\begin{proof}[Proof of Theorem \ref{thm_cond}]
We follow a similar strategy as in the proof of Proposition \ref{thm_cond_pos}. First note that since $\lim_{q\searrow 0} \kappa(q)/\hat\kappa(q)$ exists, the ratio is bounded for $q \in (0,1)$ by some $\beta>0$. Hence, with Corollary \ref{cor_asymp} we get
$$\frac{1}{\hat\kappa(q)}\mP^y(e_q < T_{[a,b]}) \leq h_+^q(y) + \frac{\kappa(q)}{\hat\kappa(q)}h_-^q(y) \leq h_+(y) + \beta h_-(y)$$
for all $y \notin [a,b]$. So we use dominated convergence and the second part of Corollary \ref{cor_asymp} to get
\begin{align*}
&\quad \lim\limits_{q \searrow 0} \mP^x(\Lambda, t < e_q | e_q<T_{[a,b]})\\
 &= \lim\limits_{q \searrow 0} \frac{1}{\mP^x(e_q < T_{[a,b]})} \int\limits_{(t,\infty)}q e^{-qs} \mP^x\left(\Lambda, s <T_{[a,b]} \right) \, \dd s  \\
&= \frac{1}{h_+(x)+Ch_-(x)} \lim\limits_{q \searrow 0} \frac{e^{-qt}}{\hat\kappa(q)} \int\limits_{(0,\infty)}q e^{-qs} \mP^x\left(\Lambda, s+t <T_{[a,b]} \right) \, \dd s \\
&= \frac{1}{h_+(x)+Ch_-(x)} \lim\limits_{q \searrow 0} \frac{e^{-qt}}{\hat\kappa(q)} \int\limits_{(0,\infty)} q e^{-qs} \mE^x\Big[ \1_\Lambda  \1_{\{ t <T_{[a,b]} \}} \mP^{\xi_t}(s<T_{[a,b]}) \Big] \, \dd s \\
&= \frac{1}{h_+(x)+Ch_-(x)} \lim\limits_{q \searrow 0} \frac{1}{\hat\kappa(q)} \mE^x\Big[ \1_\Lambda  \1_{\{ t <T_{[a,b]} \}} \mP^{\xi_t}(e_q<T_{[a,b]}) \Big]  \\
&= \frac{1}{h_+(x)+Ch_-(x)} \mE^x\Big[ \1_\Lambda  \1_{\{ t <T_{[a,b]} \}}\lim\limits_{q \searrow 0} \frac{1}{\hat\kappa(q)}  \mP^{\xi_t}(e_q<T_{[a,b]}) \Big]  \\
&= \frac{1}{h_+(x)+Ch_-(x)} \mE^x\Big[ \1_\Lambda  \1_{\{ t <T_{[a,b]} \}} \big(h_+(\xi_t) + Ch_-(\xi_t)\big) \Big].
\end{align*}
\end{proof}

\subsection{Long-time behaviour} Finally, we analyze the transience behavior of the conditioned processes constructed in the previous section.

\begin{proof}[Proof of Proposition \ref{thm_drift+}]
\textbf{Step 1:} We show that $\xi$ under $\mP^x_+$ is almost surely bounded from below. First note that, for $x<a$,
\begin{align*}
\mE^x\Big[ \1_{\{ T_{[a,\infty)} < T_{[a,b]} \}} h_+(\xi_{T_{[a,\infty)}}) \Big] &= \int\limits_{(b,\infty)} h_+(y) \, \nu_1^x(\dd y)\\
&= \sum_{k =0}^\infty\int\limits_{(b,\infty)}  \int\limits _{(b,\infty)} U_-(z-b)\, \nu_{2k} ^y(\dd z) \, \nu_1^x(\dd y) \\
&= \sum_{k =0}^\infty \int\limits_{(b,\infty)} U_-(z-b)\, \nu_{2k+1}^x(\dd z) \\
&= h_+(x).
\end{align*}
For the first equality we used $\nu^x_1(\dd y) = \mP^x(\xi_{T_{[a,\infty)}} \in \dd y, T_{[a,\infty)} < T_{[a,b]})$ for $x<a$, in the second we plugged-in the definition of $h_+(y)$ for $y>b$ and used Fubini's theorem, in the third we used \eqref{cc} and for the final equality we used the definition of $h_+(x)$ for $x<a$. Since $\xi_{T_{(-\infty,c]}}<a$ for $c<a$ it follows, for all $x \in \mR \setminus [a,b]$, that
\begin{align*}
&\quad \mP_+^x(T_{(-\infty,c]} < \infty \text{ for all } c <a)\\
&= \lim_{c \rightarrow -\infty} \mP_+^x(T_{(-\infty,c]} < \infty)\\
&=  \frac{1}{h_+(x)} \lim_{c \rightarrow -\infty} \mE^x \Big[ \1_{\{ T_{(-\infty,c]} < T_{[a,b]} \}} h_+(\xi_{T_{(-\infty,c]}}) \Big] \\
&= \frac{1}{h_+(x)} \lim_{c \rightarrow -\infty} \mE^x \Big[ \1_{\{ T_{(-\infty,c]} < T_{[a,b]} \}} \mE^{\xi_{T_{(-\infty,c]}}} \big[\1_{\{ T_{[a,\infty)} < T_{[a,b]} \}} h_+(\xi_{T_{[a,\infty)}})\big] \Big] \\
&= \frac{1}{h_+(x)} \lim_{c \rightarrow -\infty} \mE^x \Big[ \1_{\{ T_{(-\infty,c]} < T_{[a,b]} \}} \mE^{x} \big[(\1_{\{ T_{[a,\infty)} < T_{[a,b]} \}} h_+(\xi_{T_{[a,\infty)}})) \circ \theta_{T_{(-\infty,c]}} \, | \, \cF_{T_{(-\infty,c]}} \big] \Big],
\end{align*}
where we used again the strong Markov property \eqref{strongMP} with $Y = \1_{\{ T_{[a,\infty)} < T_{[a,b]} \}} h_+(\xi_{T_{[a,\infty)}})$ in the final equality.
According to Theorem 1.3.6 of Chung and Walsh \cite{Chu_Wal_01} it holds that
$$\{ T_{(-\infty,c]} < T_{[a,b]} \} \in \cF_{ T_{(-\infty,c]}} \cap \cF_{T_{[a,b]}} \subseteq \cF_{ T_{(-\infty,c]}}.$$
So we continue for all $x \in \mR \setminus [a,b]$ with
\begin{align*}
&\quad\mE^x \Big[ \1_{\{ T_{(-\infty,c]} < T_{[a,b]} \}} \mE^{x} \big[(\1_{\{ T_{[a,\infty)} < T_{[a,b]} \}} h_+(\xi_{T_{[a,\infty)}})) \circ \theta_{T_{(-\infty,c]}} \, | \, \cF_{T_{(-\infty,c]}} \big] \Big]\\
&=  \mE^x \Big[ \mE^{x} \big[\1_{\{ T_{(-\infty,c]} < T_{[a,b]} \}}(\1_{\{ T_{[a,\infty)} < T_{[a,b]} \}} h_+(\xi_{T_{[a,\infty)}})) \circ \theta_{T_{(-\infty,c]}} \, | \, \cF_{T_{(-\infty,c]}} \big] \Big] \\
&= \mE^x \Big[ \1_{\{ T_{(-\infty,c]} < T_{[a,b]} \}}\big((\1_{\{ T_{[a,\infty)} < T_{[a,b]} \}} h_+(\xi_{T_{[a,\infty)}})) \circ \theta_{T_{(-\infty,c]}} \big) \Big].
\end{align*}
Now consider just $x<a$ and observe
\begin{align*}
&\quad\mE^x \Big[ \1_{\{ T_{(-\infty,c]} < T_{[a,b]} \}}\big((\1_{\{ T_{[a,\infty)} < T_{[a,b]} \}} h_+(\xi_{T_{[a,\infty)}})) \circ \theta_{T_{(-\infty,c]}} \big) \Big]\\
&= \sum_{k=0}^\infty \mE^{x} \Big[\1_{\{ T_{(-\infty,c]} \in [\tilde\tau_{2k},\tilde\tau_{2k+1})\}}\big((\1_{\{ T_{[a,\infty)} < T_{[a,b]} \}} h_+(\xi_{T_{[a,\infty)}})) \circ \theta_{T_{(-\infty,c]}} \big) \Big] \\
&= \sum_{k=0}^\infty \mE^{x} \Big[\1_{\{ T_{(-\infty,c]} \in [\tilde\tau_{2k},\tilde\tau_{2k+1})\}} \1_{\{ \tilde\tau_{2k+1} < T_{[a,b]} \}} h_+(\xi_{\tilde\tau_{2k+1}}) \Big]\\
&= \sum_{k=0}^\infty \mE^{x} \Big[\1_{\{ T_{(-\infty,c]} \in [\tau_{2k},\tau_{2k+1})\}} \1_{\{ \tau_{2k+1} < T_{[a,b]} \}} h_+(\xi_{\tau_{2k+1}}) \Big],
\end{align*}
where $\tilde\tau_k = \min(\tau_k,T_{[a,b]})$ as in the proof of Proposition \ref{prop_asymp}. Combining the above computations gives
\begin{align} \label{eq_help1}
&\quad\mP_+^x(T_{(-\infty,c]} < \infty \text{ for all } c <a)\\
&= \frac{1}{h_+(x)} \lim_{c \rightarrow -\infty} \sum_{k=0}^\infty \mE^{x} \Big[ \1_{\{ T_{(-\infty,c]} \in [\tau_{2k},\tau_{2k+1})\}} \1_{\{ \tau_{2k+1} < T_{[a,b]} \}} h_+(\xi_{\tau_{2k+1}}) \Big] \nonumber
\end{align}
for $x<a$. Our aim is to switch the limit and the sum. In order to justify the dominated convergence theorem it is enough to verify
$$\sum_{k=0}^\infty \mE^{x} \la \1_{\{ \tau_{2k+1} < T_{[a,b]} \}} h_+(\xi_{\tau_{2k+1}}) \ra < \infty.$$
With Proposition \ref{lemma_finite} we have 
\begin{align*}
&\quad \mE^{x} \Big[ \1_{\{ \tau_{2k+1} < T_{[a,b]} \}} h_+(\xi_{\tau_{2k+1}}) \Big] \\
&\leq c_1 \mE^{x} \Big[ \1_{\{ \tau_{2k+1} < T_{[a,b]} \}} U_-(\xi_{\tau_{2k+1}}-b) \Big] + c_3 \mP^x(\tau_{2k+1} < T_{[a,b]}) \\
&\leq c_1 \mE^{x} \Big[ \1_{\{ K^\dag \geq 2k+1 \}} U_-(\xi_{\tau_{2k+1}}-b) \Big] + c_3 \nu_{2k+1}^x((b,\infty)) \\
&\leq c_1 \mE^{x} \Big[ \1_{\{ K^\dag \geq 2k+1 \}} U_-(\xi_{\tau_{2k+1}}-b) \Big] + c_3 \gamma^{2k}
\end{align*}
where $c_1,c_3$ and $\gamma$ are the constants from Proposition \ref{lemma_finite} and its proof. It follows that
\begin{align*}
&\quad\sum_{k=0}^\infty \mE^{x} \big[ \1_{\{ \tau_{2k+1} < T_{[a,b]} \}} h_+(\xi_{\tau_{2k+1}}) \big] \\
&\leq c_1 \sum_{k=0}^\infty  \mE^{x} \big[ \1_{\{ K^\dag \geq 2k+1 \}} U_-(\xi_{\tau_{2k+1}}-b) \big] + c_3 \sum_{k=0}^\infty \gamma^{2k}\\
&= c_1 h_+(x) + \frac{c_3}{1-\gamma^2}<\infty.
\end{align*}
So we can switch the limit and the integral in \eqref{eq_help1}. With the same upper bound for every summand for itself we can even move the limit inside the expectation. Hence,
\begin{align*}
&\quad \mP_+^x(T_{(-\infty,c]} < \infty \text{ for all } c <a)\\
&= \frac{1}{h_+(x)} \sum_{k=0}^\infty \mE^{x} \la \lim_{c \rightarrow -\infty} \1_{\{ T_{(-\infty,c]} \in [\tau_{2k},\tau_{2k+1})\}} \1_{\{ \tau_{2k+1} < T_{[a,b]} \}} h_+(\xi_{\tau_{2k+1}}) \ra.
\end{align*}
Since $\xi$ oscillates (which implies $\tau_k<\infty$ $\mP^x$-almost surely) we obtain that $\1_{\{ T_{(-\infty,c]} \in [\tau_{2k},\tau_{2k+1}) \}}$ converges to $0$ almost surely under $\mP^x$ for $c \rightarrow -\infty$. Hence,
$$\mP_+^x(T_{(-\infty,c]} < \infty \text{ for all } c <a) = 0$$
for $x<a$. For $x>b$ it is proved analogously that
\begin{align*}
&\quad \mP_+^x(T_{(-\infty,c]} < \infty \text{ for all } c <a)\\
&= \frac{1}{h_+(x)} \lim_{c \rightarrow -\infty} \sum_{k=0}^\infty \mE^{x} \big[\1_{\{ T_{(-\infty,c]} \in [\tau_{2k+1},\tau_{2k+2})\}} \1_{\{ \tau_{2k+2} < T_{[a,b]} \}} h_+(\xi_{\tau_{2k+2}})) \big] 
\end{align*}
and, with the above argumentation, we also find that $\mP_+^x(T_{(-\infty,c]} < \infty \text{ for all } c <a) = 0$ for $x>b$. This finishes the arguments for Step 1.\\ \smallskip

\textbf{Step 2:} In the second step we show that $\xi$ is transient under $\mP^x_+$, i.e. only spends finite time in sets of the form $[d,a)\cup(b,c]$ for $d<a$ and $c>b$. Actually, we even show that the expected occupation is finite:
\begin{align}\label{eq_trans}
\begin{split}
&\quad \mE_+^x \Big[ \int\limits_{[0,\infty)} \1_{\{ \xi_t \in [d,a)\cup(b,c] \}} \dd t \Big] \\
&=  \int\limits_{[0,\infty)} \mP_+^x(\xi_t \in [d,a)\cup(b,c]) \, \dd t \\
&= \int\limits_{[0,\infty)} \mE^x \Big[ \1_{\{\xi_t \in [d,a)\cup(b,c] \}} \1_{\{ t < T_{[a,b]} \}} \frac{h_+(\xi_t)}{h_+(x)} \Big] \, \dd t  \\
&\leq \frac{1}{h_+(x)} \sup\limits_{y \in [d,a)\cup(b,c]} h_+(y) \int\limits_{[0,\infty)} \mE^x \Big[ \1_{\{\xi_t \in [d,a)\cup(b,c] \}} \1_{\{ t < T_{[a,b]} \}} \Big] \, \dd t. 
\end{split}
\end{align}
Recalling Proposition \ref{lemma_finite}, $\sup_{y \in [d,a)\cup(b,c]} h_+(y)$ is finite and it remains to show finiteness of
$$\int\limits_{[0,\infty)} \mE^x \Big[ \1_{\{\xi_t \in [d,a)\cup(b,c] \}} \1_{\{ t < T_{[a,b]} \}} \Big] \, \dd t$$
which is just the potential of $[d,a)\cup(b,c]$ of the process killed on entering $[a,b]$. To abbreviate we denote the potential of $(\xi,\mP^x)$ killed on entering a Borel set $B$ by $U^B(x,\dd y)$. It follows
\begin{align*}
U^{[a,b]}(x,[d,a)\cup(b,c]) = \sum\limits_{k=0}^\infty \left(U^{(-\infty,b]}(\nu_{2k}^x,(b,c]) + U^{[a,\infty)}(\nu_{2k+1}^x,[d,a)) \right).
\end{align*}
To compute the righthand side we apply Proposition VI.20 of Bertoin \cite{Bert_01} for $y >b$:
\begin{align*}
U^{(-\infty,b]}(y,(b,c]) &= U^{(-\infty,0]}(y-b,(0,c-b]) \\
&= \int\limits_{(0,c-b]} \int\limits_{[(y-b-u)^+,y-b]} \, U_+(\dd u +v-(y-b)) \,U_-(\dd v) \\
&= \int\limits_{[0,y-b]} \Big( \int\limits_{(0,c-b]} \1_{\{ u \geq y-b-v \}} \, U_+(\dd u -(y-b-v)) \Big) \, U_-(\dd v) \\
&= \int\limits_{[0,y-b]} U_+(c+v-y)  \,U_-(\dd v) \\
&\leq U_+(c-b) U_-(y-b).
\end{align*}
It holds analogously that $U^{[a,\infty)}(y,[d,a)) \leq U_-(a-d) U_+(a-y)$ for $y > a$. So we have
\begin{align*}
U_{[a,b]}(x,[d,a)\cup(b,c]) &\leq U_+(c-b) \sum\limits_{k=0}^\infty \int\limits_{(b,\infty)} U_-(y-b) \, \nu_{2k}^x(\dd y)\\
& \quad + U_-(a-d) \sum\limits_{k=0}^\infty \int\limits_{(-\infty,a)} U_+(a-y) \, \nu_{2k+1}^x(\dd y) \\
&=U_+(c-b)h_+(x) + U_-(a-d))h_-(x)<\infty.
\end{align*}

It follows in particular that the time the process $(\xi,\mP^x_+)$ spends in sets of the form $[d,a)\cup(b,c]$ is finite almost surely. Together with the first result that the process is bounded below almost surely and that the process is conservative it follows that $\lim_{t \rightarrow \infty} \xi_t = +\infty$ almost surely under $\mP^x_+$.
\end{proof}

\begin{proof}[Proof of Theorem \ref{thm_drifth}]
The proof strategy is similar to the one above. Transience of the conditioned process is verified again by computing the occupation measure using the representation of the conditioned process as $h$-transform. The computation is in analogy to (\ref{eq_trans}), using that $h=h_++C h_-$ is bounded by Proposition \ref{lemma_finite}.\smallskip
%Since bounded from below (resp. above) combined with transience implies divergence to $+\infty$ (resp. $-\infty$) we compute the probabilities of one-sided boundedness.\\ \smallskip

Next, recall from the counterpart of Proposition \ref{thm_drift+} for $\mP^x_-$ that under ($\hat B$),
\begin{align*}
	%\mP_{+}^x(T_{[c,-\infty)}<\infty)=1\quad\text{and}
	\quad \mP^x_-(T_{(-\infty,c]}<\infty) = 1, \quad c<a
\end{align*}
for all $x \in \mR\setminus [a,b]$. Since \eqref{eq_htrafo_stoppingtime} implies 
$$\mP^x_-(T_{(-\infty,c]}<\infty) = \frac{1}{h_-(x)}\mE^x \la \1_{\{ T_{(-\infty,c]} < T_{[a,b]} \}} h_-(\xi_{T_{(-\infty,c]}}) \ra$$
we deduce
\begin{align}\label{help}
\mE^x \la \1_{\{ T_{(-\infty,c]} < T_{[a,b]} \}} h_-(\xi_{T_{(-\infty,c]}}) \ra = h_-(x), \quad c<a
\end{align}
for all $x \in \mR\setminus [a,b]$ under ($\hat B$). If ($\hat B$) fails we know
$$h_-(x) =  \begin{cases}
0 \quad &\text{if } x>b \\
U_+(a-x) \quad &\text{if } x <a
\end{cases}.$$
Let us check if \eqref{help} holds in this case, too. If $x>b$ the left-hand side of \eqref{help} is $0$ (because there are no jumps bigger than $b-a$), as well as the right-hand side. For $x>a$ the measure $\mP_-^x$ corresponds to the process conditioned to stay below $a$ which is known to drift to $-\infty$ (see Chaumont and Doney \cite{Chau_Don_01}). In particular it holds 
$$\mP^x_-(T_{(-\infty,c]}<\infty) = 1, \quad c<a$$
from which we can deduce \eqref{help} in the same way as before. So \eqref{help} holds for all $x \in \mR\setminus [a,b]$ just under (A).\smallskip

Again using \eqref{eq_htrafo_stoppingtime} yields
\begin{align*}
&\quad \mP_\updownarrow^x(T_{(-\infty,c]} <\infty)\\
&= \frac{1}{h(x)} \Big( \mE^x \big[ \1_{\{ T_{(-\infty,c]} < T_{[a,b]} \}} h_+(\xi_{T_{(-\infty,c]}}) \big] + \mE^x \big[ \1_{\{ T_{(-\infty,c]} < T_{[a,b]} \}} Ch_-(\xi_{T_{(-\infty,c]}}) \big] \Big)\\
 &= \frac{1}{h(x)} \mE^x \big[ \1_{\{ T_{(-\infty,c]} < T_{[a,b]} \}} h_+(\xi_{T_{(-\infty,c]}}) \big] + \frac{Ch_-(x)}{h(x)}.
\end{align*}
In the proof of Proposition \ref{thm_drift+} we have already seen that $\mE^x \big[ \1_{\{ T_{(-\infty,c]} < T_{[a,b]} \}} h_+(\xi_{T_{(-\infty,c]}}) \big]$ vanishes for $c \rightarrow -\infty$, hence,
$$\mP_\updownarrow^x(\xi \text{ is unbounded below}) = \mP_\updownarrow^x(T_{(-\infty,c]} <\infty \text{ for all } c<a)= \frac{Ch_-(x)}{h(x)}.$$
So we get 
$$\mP_\updownarrow^x(\xi \text{ is bounded below}) = 1-\frac{Ch_-(x)}{h(x)} = \frac{h_+(x)}{h(x)}$$
and, because of transience,
$$\frac{h_+(x)}{h(x)}=\mP_\updownarrow^x(\xi \text{ is bounded below}) = \mP_\updownarrow^x(\lim_{t \rightarrow \infty} \xi_t = \infty).$$
Analogously one derives $\mP_\updownarrow^x(\lim_{t \rightarrow \infty} \xi_t = \infty)= \frac{Ch_-(x)}{h(x)}$ and the proof is complete.
%From the definition of $h$ and the transience 
%\begin{align*}
%&\quad \mP_\updownarrow^x(\xi \text{ is unbounded below}) + \mP_\updownarrow^x(\xi \text{ is unbounded above})\\
%&= \frac{h_+(x) + Ch_-(x)}{h(x)}\\
%&= 1 \\
%&= \mP_\updownarrow^x(\xi \text{ is unbounded}).
%\end{align*}
%Since $\mP_\updownarrow^x(A \cap B) = \mP_\updownarrow^x(A) +  \mP_\updownarrow^x(B) - \mP_\updownarrow^x(A \cup B)$ for any events $A$ and $B$ it follows that
%$$\mP_\updownarrow^x(\xi \text{ is unbounded below and above}) = 0.$$
%Using this and that $(\xi,\mP^x_\updownarrow)$ is transient we obtain
%\begin{align*}
%\mP_\updownarrow^x(\lim\limits_{t \rightarrow \infty} \xi_t = -\infty) = \mP_\updownarrow^x(\xi \text{ is unbounded below}) = \frac{Ch_-(x)}{h(x)}.
%\end{align*}
\end{proof}

\section{Extension to transient Lévy processes}

When conditioning a process to avoid an interval, the most interesting case is
when the process is recurrent; if it is transient, it may avoid the interval with
positive probability, and things become simpler. On the other hand, the conditionings
in Proposition \ref{thm_cond_pos}, to avoid the interval while finishing above (or below) it,
may still be non-trivial.
In this section, we drop Assumption $(A)$, and require only that $\xi$ is not a compound
Poisson process and does not oscillate. In particular, we do not assume that $\xi$ has finite
second moments; only for the study of $h_-$ do we need further conditions.
\smallskip

Without loss of generality, we assume from now on that $\xi$ drifts to $+\infty$, and indicate which of our results still hold and which need modification.
Under this assumption, the function $h$ defined by \eqref{h} simplifies to $h_+$. This can be seen from the fact that $\kappa(0) = 0 < \hat\kappa(0)$, which implies $C=\lim_{q\searrow 0}\frac{\kappa(q)}{\hat\kappa(q)} = 0$. 

\subsection{Study of \texorpdfstring{$h=h_+$}{h = h\textunderscore +}}

For the study of $h$ (which is now equal to $h_+$) we need to distinguish two cases
based on whether or not condition $(B)$ is satisfied.

\subsubsection{Condition \texorpdfstring{$(B)$}{(B)} holds}

Since the L\'evy process is transient, the event $\{T_{[a,b]} = \infty\}$ has positive probability for every starting point. The conditioning simplifies dramatically and our results are still valid,
as we now demonstrate. Let $\ell(x):= \mP^x(T_{[a,b]} = \infty)$ for $x\notin [a,b]$. This is easily seen to be harmonic using the strong Markov property:
\begin{align}\label{eq_prob_harm}
\begin{split}
\mE^x\bigl[ \1_{\{t < T_{[a,b]}\}} \ell(\xi_t)\bigr]
&=\mE^x \big[ \1_{\{ t < T_{[a,b]} \}} \mP^{\xi_t}(T_{[a,b]} = \infty) \big] \\
&= \lim_{s \rightarrow \infty} \mE^x \big[ \1_{\{ t < T_{[a,b]} \}} \mP^{\xi_t}(T_{[a,b]} >s) \big] \\
&=\lim_{s \rightarrow \infty} \mP^{x}(T_{[a,b]} > t+s)  \\
&= \mP^{x}(T_{[a,b]} = \infty).
\end{split}
\end{align}
Transience ensures that $\ell$ is a positive harmonic function. We next show that $\ell$ is indeed a multiple of $h=h_+$. To do so we will use the identity $\hat{\kappa}(q)U_-^q(x) = \mP^x(e_q<T_{(-\infty,0]}),$ where $e_q$ is an independent exponentially distributed random variable with parameter $q>0$ (see Kyprianou \cite{Kyp_01}, Section 13.2.1 for a general Lévy process). Since $\xi$ drifts to $+\infty$, we have $\hat{\kappa}(0)>0$, and hence
$$\hat{\kappa}(0)U_-(x) = \mP^x(T_{(-\infty,0]}=\infty), \quad x >0.$$
The idea is to separate the two-sided entrance problem in infinitely many one-sided entrance problems and use the strong Markov property to combine them. For $x>b$, using the strong Markov property, we find
\begin{align*}
&\quad \mP^x(T_{[a,b]} = \infty)\\
&=  \mP^x(T_{(-\infty,b]} =\infty) + \mP^x(T_{[a,b]} = \infty, T_{(-\infty,b]} <\infty) \\
&=  \mP^x(T_{(-\infty,b]} =\infty) + \mE^x \Big[ \1_{\{ T_{(-\infty,b]} <\infty, \xi_{T_{(-\infty,b]}<a} \}} \mP^{\xi_{T_{(-\infty,b]}}}(T_{[a,b]}=\infty) \Big] \\
&=  \hat{\kappa}(0) U_-(x-b) + \mE^x \Big[ \1_{\{ T_{(-\infty,b]} <\infty, \xi_{T_{(-\infty,b]}<a} \}} \mE^{\xi_{T_{(-\infty,b]}}}\big[ \1_{\{ \xi_{T_{[a,\infty)}} >b \}} \mP^{\xi_{T_{[a,\infty)}}}(T_{[a,b]}=\infty) \big] \Big] \\
&=  \hat{\kappa}(0) U_-(x-b) + \int\limits_{(b,\infty)} \mP^{y}(T_{[a,b]}=\infty) \, \nu_2^x(\dd y).
\end{align*}
Now we split up $\mP^{y}(T_{[a,b]}=\infty)$ in the same manner, i.e.,
$$\mP^{y}(T_{[a,b]}=\infty) = \hat{\kappa}(0) U_-(y-b) + \int\limits_{(b,\infty)} \mP^{z}(T_{[a,b]}=\infty) \, \nu_2^x(\dd z).$$
Using $\int\limits_{(b,\infty)} \nu^z_{2}(\dd y) \, \nu_{2}^x(\dd z) = \nu^x_{4}(\dd y)$ from \eqref{cc} yields
\begin{align*}
&\quad \mP^x(T_{[a,b]} = \infty)\\
&= \hat{\kappa}(0) \Big( U_-(x-b) + \int\limits_{(b,\infty)} U_-(y-b) \, \nu_2^x(\dd y) \Big) + \int\limits_{(b,\infty)} \mP^{y}(T_{[a,b]}=\infty) \, \nu_4^x(\dd y).
\end{align*}
By induction the following series representation is obtained:
$$\mP^x(T_{[a,b]} = \infty) = \hat{\kappa}(0) \sum_{k=0}^\infty \int\limits_{(b,\infty)} U_-(y-b) \, \nu_{2k}^x(\dd y).$$
For $x<a$ a similar computation can be carried out, and we obtain
\begin{align*}
\ell(x) = \mP^x(T_{[a,b]} = \infty) &= \begin{cases}
\hat{\kappa}(0) \sum\limits_{k=0}^\infty \int\limits_{(b,\infty)} U_-(y-b) \, \nu_{2k}^x(\dd y)&\quad \text{if } x>b \\
\hat{\kappa}(0) \sum\limits_{k=0}^\infty \int\limits_{(b,\infty)} U_-(y-b) \, \nu_{2k+1}^x(\dd y) &\quad \text{if } x<a
\end{cases}\\
&= \hat{\kappa}(0) h_+(x) = \hat\kappa(0) h(x).
\end{align*}

\smallskip

\textbf{Theorem \ref{thm_harmonic}:}
This is a consequence of the discussion above.\smallskip 

\textbf{Theorem \ref{thm_cond}:} Since we condition here on a positive probability event, the $h$-transform and the conditioning are related in a standard way, using the strong Markov property and integrating out $e_q$:
\begin{align*}
\mE^x \Big[ \1_{\Lambda} \1_{\{ t < T_{[a,b]} \}} \frac{\ell(\xi_t)}{\ell(x)} \Big] &= \frac{1}{ \mP^{x}(T_{[a,b]} = \infty)} \mE^x  \Big[\1_{\Lambda} \1_{\{ t < T_{[a,b]} \}} \mP^{\xi_t}(T_{[a,b]} = \infty) \Big] \\
&= \lim_{q \searrow 0} \frac{1}{ \mP^{x}(e_q <T_{[a,b]})} \mE^x \Big[\1_{\Lambda} \1_{\{ t < T_{[a,b]} \}} \mP^{\xi_t}(e_q < T_{[a,b]}) \Big]\\
&= \lim_{q \searrow 0} \frac{ \mP^{x}(\Lambda, t+e_q < T_{[a,b]}) }{ \mP^{x}(e_q <T_{[a,b]})}\\
&= \lim_{q \searrow 0} \frac{e^{qt} \mP^{x}(\Lambda, t<e_q < T_{[a,b]}) }{ \mP^{x}(e_q <T_{[a,b]})}\\
&= \lim_{q \searrow 0} \mP^x(\Lambda, t< e_q \,|\, e_q < T_{[a,b]}),
\end{align*}
for $\Lambda \in \cF_t$, $t \geq 0$.\smallskip

\textbf{Proposition \ref{thm_cond_pos}:} The conditioning of Proposition \ref{thm_cond_pos} is equivalent to the conditioning of Theorem \ref{thm_cond}, since the additional condition to stay above the interval at late time vanishes in the limit due to the transience towards $+\infty$. Since $h=h_+$ the result of Proposition \ref{thm_cond_pos} follows.\smallskip

\textbf{Proposition \ref{thm_drift+} and Theorem \ref{thm_drifth}:} Since the
conditioned measure is a restriction of the original one,
the long-time behaviour of the conditioned process is identical to that of the original process. Hence, the statements of Proposition \ref{thm_drift+} and Theorem \ref{thm_drifth} hold.
\subsubsection{Condition $(B)$ fails}

The definition of $h_+$ in this case simplifies to 
\[
 h_+(x) =
  \begin{cases}
    U_-(x-b) & \text{if } x > b \\
    0 &  \text{if } x< a
  \end{cases}.
\]
This function is plainly not positive everywhere. It is nonetheless harmonic
for the process killed on entering $[a,b]$. The conditionings in Theorem \ref{thm_cond}
and Proposition \ref{thm_cond_pos} can still be carried out but, as we now prove,
the results are somewhat
different.
\smallskip

Let $h_{\uparrow}: (b,\infty) \to [0,\infty)$ be given by $h_{\uparrow}(x) = U_-(x-b)$,
the restriction of $h_+$ to $(b,\infty)$. As shown by Chaumont and Doney
\cite{Chau_Don_01}, this function is harmonic for the process $\xi$
killed on entering $(-\infty,b]$, and the $h$-transform of this process
using $h_{\uparrow}$ is the process $\xi$ conditioned to avoid $(-\infty,b]$.
We will write $(\mP^x_\uparrow)_{x\in (b,\infty)}$ for the probabilities
associated with this Markov process.
\smallskip

Consider now the conditioning of \textbf{Proposition \ref{thm_cond_pos}}. When $x > b$ the process cannot cross below the set $[a,b]$ and return above it
without hitting the set. Therefore, we have that
\[
  \lim_{q\searrow 0} \mP^x(\Lambda, t<e_q \mid e_q < T_{[a,b]}, \xi_{e_q} > b)
  = \lim_{q\searrow 0} \mP^x(\Lambda, t<e_q \mid e_q < T_{(-\infty,b]})
  = \mP^x_\uparrow(\Lambda),
\]
the last equality being due to Chaumont and Doney \cite{Chau_Don_01}. For
$x<a$, $\mP^x(e_q<T_{[a,b]},\xi_{e_q}>b)=0$ for every $q> 0$, so the
conditioning does not have any sense.
In total, the conditioning of Proposition \ref{thm_cond_pos}
reduces to conditioning $\xi$ to avoid $(-\infty,b)$.
\smallskip

We turn next to the conditioning in \textbf{Theorem \ref{thm_cond}}.
Let us define $h_{\downarrow}: (-\infty,a) \to [0,\infty)$ by
$h_{\downarrow}(x) = U_+(a-x)$, which is a positive harmonic function
for the process killed on entering $[a,\infty)$ resulting in the process
conditioned to avoid $[a,\infty)$ when $h$-transformed with $h_\downarrow$. As before, we write $(\mP^x_{\downarrow})_{x\in (-\infty,a)}$
for the probabilities associated with the conditioned process,
which is killed at its lifetime $\zeta$.
By the same reasoning in the case where $(B)$ holds,
$\lim_{q\searrow 0}\mP^x(T_{[a,b]}>e_q) = \hat{\kappa}(0) h_+ (x)=\hat{\kappa}(0) h_\uparrow (x)$ when $x>b$;
and, when $x < a$, using the asymptotics of $T_{[a,\infty)}$ which 
we have already seen, we obtain
$\mP^x(T_{[a,b]}>e_q) = \mP^x(T_{[a,\infty)}>e_q) \sim \kappa(q) U_+(a-x)$
as $q \searrow 0$, since $\xi$ cannot jump over $[a,b]$ from below.
If $x>b$, and $\Lambda\in \cF_t$,
the same technique as in the proof of Theorem \ref{thm_cond} gives rise to the
calculation
  \begin{align*}
    &\quad
    \lim_{q \searrow 0} \mP^x(\Lambda, t<e_q \mid e_q< T_{[a,b]}) \\
    &= \frac{1}{\hat\kappa(0)h_\uparrow(x)}\mE^x\bigl[\1_\Lambda \1_{\{t<T_{[a,b]}\}}
    \lim_{q\searrow 0} \mP^{\xi_t}(e_q<T_{[a,b]}) \bigr]
    \\
    &= \frac{1}{\hat\kappa(0)h_\uparrow(x)}\mE^x\Bigl[\1_\Lambda \1_{\{t<T_{[a,b]}\}}
    \lim_{q\searrow 0}
    \bigl( \1_{\{t<T_{(-\infty,b]}\}}\hat\kappa(0) h_\uparrow(\xi_t)
     + \1_{\{t > T_{(-\infty,b]}\}} \kappa(q) U_+(a-\xi_t) \bigr )\Bigr] \\
    &= \frac{1}{h_\uparrow(x)} \mE^x[ h_+(\xi_t) \1_\Lambda \1_{\{t<T_{(-\infty,b]}\}}]
    = \mP^x_\uparrow(\Lambda).
  \end{align*}
Similarly, if $x<a$, we obtain
  $\lim_{q \searrow 0} \mP^x(\Lambda, t<e_q \mid e_q< T_{[a,b]})
  = \mP_\downarrow^x(\Lambda, t < \zeta)$.
\smallskip
  
This shows that the conditioning from Theorem \ref{thm_cond} leads
not to a single Doob $h$-transform of a killed Lévy process, but rather to
a Markov process which behaves entirely differently depending on whether
it is started above or below the interval. The long-time 
behaviour can be deduced from Chaumont and Doney \cite{Chau_Don_01}:
the conditioned process approaches $+\infty$ when started above $b$,
and is killed  when started below $a$.

\subsection{Study of \texorpdfstring{$h_-$}{h\textunderscore -}}
This section is kept informal; the claims can be proved by an adaptation of arguments developed in Section \ref{sec_proofs}.\smallskip

 In order to study $h_-$ we need to assume that $\mE[H_1]<\infty$ and $\hat\mE[H_1] <\infty$. Note that here the descending ladder height subordinator has finite lifetime $\zeta$,
so we understand
$\hat\mE[H_1] = \hat\mE[H_1 \1_{1<\zeta}]$. The function $h_-$ is merely superharmonic,
in the sense that
	\[ \mE^x[h_-(\xi_t) \1_{\{ t<T_{[a,b]} \}}] \le h_-(x),\quad x \in \mR\setminus [a,b].
	\]
We may still define the superharmonic transform
\[ \mP^x_-(\Lambda, t < \zeta) = \mE^x\left[ \1_\Lambda \1_{\{ t<T_{[a,b]} \}} \frac{h_-(\xi_t)}{h_-(x)}\right],
  \quad x \in \mR\setminus [a,b],
\]
but the transformed process is now a killed Markov process, with lifetime $\zeta$.
\smallskip

The dual version of the conditioning of \textbf{Proposition \ref{thm_cond_pos}} is then
given by
\begin{equation}\label{eq_subMarkov}
  \mP^x_{-}(\Lambda, t < \zeta)
  = \lim_{q\searrow 0} \mP^x(\Lambda, t<e_q \mid e_q < T_{[a,b]}, \xi_{e_q}<a),\quad x \in \mR\setminus [a,b],
\end{equation}
and gives rise to a killed strong Markov process. This is a generalization of the subordinator conditioned to stay below a level as studied in Kyprianou et al.\ \cite{Kyp_Riv_Sen_02}.

\bibliographystyle{abbrvnat}
\bibliography{references}

\end{document}

%% file: basis_paper.tex
\newcommand{\mE}{\mathbb{E}}
\newcommand{\mN}{\mathbb{N}}
\newcommand{\mP}{\mathbb{P}}
\newcommand{\mR}{\mathbb{R}}

\newcommand{\cF}{\mathcal{F}}

\newcommand{\la}{\left\lbrack}
\newcommand{\ra}{\right\rbrack}
\newcommand{\1}{\mathds{1}}
\newcommand\dd{\mathrm{d}}

\newtheorem {definition}{Definition}[section]
\newtheorem {theorem}[definition]{Theorem}
\newtheorem{proposition}[definition]{Proposition}
\newtheorem {lemma}[definition]{Lemma}
\newtheorem {corollary}[definition]{Corollary}

\newtheorem {remark}[definition]{Remark}

\DeclareMathOperator*{\wlim}{w-lim}